\documentclass[12pt]{article}
\input epsf.tex

\usepackage{amsmath}
\usepackage{amsthm}
\usepackage{amsfonts}
\usepackage{amssymb}
\usepackage{graphicx}
\usepackage{latexsym}

\usepackage{amsmath,amsthm,amsfonts,amssymb}

\usepackage{epsfig}

\usepackage{amssymb}
\usepackage[all]{xy}
\xyoption{poly}
\usepackage{fancyhdr}
\usepackage{wrapfig}

\theoremstyle{plain}
\newtheorem{thm}{Theorem}[section]
\newtheorem{prop}[thm]{Proposition}
\newtheorem{lem}[thm]{Lemma}
\newtheorem{cor}[thm]{Corollary}

\theoremstyle{definition}
\newtheorem{defn}{Definition}
\theoremstyle{remark}
\newtheorem{remark}{Remark}
\newtheorem{example}{Example}

\advance \topmargin by -\headheight
\advance \topmargin by -\headsep
\evensidemargin \oddsidemargin
\def\cA{{\cal A}}

\def\tbeta{\tilde \beta}

\def\cB{{\mathcal B}}

\def\cc{{\curvearrowright}}

\def\bE{{\mathbb E}}

\def\teta{{\tilde \eta}}

\def\cF{{\mathcal F}}

\def\cG{{\mathcal G}}

\def\bh{{\bar h}}

\def\sK{\mathbb K}

\def\sL{{\mathbb L}}
\def\tlambda{{\tilde \lambda}}

\def\N{{\mathbb N}}

\def\cP{{\mathcal P}}

\def\sym{{\textrm{Sym}}}

\def\hsigma{{\hat \sigma}}

\def\ttheta{\tilde \theta}

\def\chix{{\raise.5ex\hbox{$\chi$}}}

\def\bX{{\overline X}}

\def\bY{{\overline Y}}

\def\Z{{\mathbb Z}}

\def\bZ{{\overline Z}}

\begin{document}
\title{Sofic entropy and amenable groups }
\author{Lewis Bowen\footnote{supported in part by NSF grant DMS-0968762 and NSF CAREER Award DMS-0954606.} \\ Texas A\&M University \\ {\it In memory of Dan Rudolph}}



\begin{abstract}
In previous work, the author introduced a measure-conjugacy invariant for sofic group actions called sofic entropy. Here it is proven that the sofic entropy of an amenable group action equals its classical entropy. The proof uses a new measure-conjugacy invariant called upper-sofic entropy and a theorem of Rudolph and Weiss for the entropy of orbit-equivalent actions relative to the orbit change $\sigma$-algebra.
\end{abstract}

\maketitle
\noindent
{\bf Keywords}: sofic groups, amenability, entropy\\
{\bf MSC}:37A35\\

\noindent
\tableofcontents

\section{Introduction}
The paper [Bo10a] introduced a family of measure-conjugacy invariants referred to as sofic entropy for actions of sofic groups. This entropy is inspired by the classical Kolmogorov-Sinai entropy and shares many of its features. The main goal of this paper is to show that the sofic entropy of an amenable group action equals its classical entropy. An alternative approach based on operator algebras is being developed by Kerr and Li [KL1, KL2]. The reader is encouraged to review [Bo10a] for more background.



\subsection{Sofic groups} 

To begin, let us recall the definition of a sofic group. 

\begin{defn}[Sofic groups]\label{defn:sofic}
Let $G$ be a countable group. For any integer $m_i>0$, let $[m_i]=\{1,\ldots, m_i\}$ and $\sym(m_i)$ denote the symmetric group on $[m_i]$. Let $\Sigma=\{\sigma_i\}_{i=1}^\infty$ be a sequence of maps $\sigma_i:G \to \sym(m_i)$ which are not assumed to homomorphisms. Then $\Sigma$ is a {\em sofic approximation} to $G$ if for every $g,h \in G$,
$$\lim_{i\to\infty} \frac{1}{m_i}\#\{p \in [m_i]:~\sigma(g)\sigma(h)p = \sigma(gh)p\} = 1$$
and for every $g\ne h \in G$,
$$\lim_{i\to\infty} \frac{1}{m_i}\#\{p \in [m_i]:~\sigma(g)p \ne \sigma(h)p\} = 1.$$
To avoid trivialities, we also assume $\lim_{i\to\infty} m_i = +\infty$, which is necessarily true if $G$ is infinite. $G$ is {\em sofic} if there exists a sofic approximation to $G$. 
\end{defn}

\begin{example}
If $G$ is residually finite then there exists a decreasing sequence $\{N_i\}_{i=1}^\infty$ of finite-index normal subgroups of $G$ with  $\cap_i N_i = \{e\}$. Let $\sigma_i: G \to \sym(G/N_i)$ be the canonical homomorphism given by the action of $G$ on $G/N_i$. Then $\{\sigma_i\}_{i=1}^\infty$ is a sofic approximation to $G$.
\end{example}

\begin{example}
If $G$ is amenable then there exists an increasing sequence $\{F_i\}_{i=1}^\infty$ of finite subsets of $G$ such that $\bigcup_i F_i = G$ and for every finite $K \subset G$
$$\lim_{i\to\infty} \frac{|KF_i \Delta F_i|}{|F_i|} = 1.$$
Let $\sigma_i:G \to \sym(F_i)$ be any map such that if $f\in F_i$, $g\in G$ and $gf \in F_i$ then $\sigma_i(g)f=gf$. Then $\{\sigma_i\}_{i=1}^\infty$ is a sofic approximation to $G$.
\end{example}

Sofic groups were defined implicitly by Gromov in [Gr99] and explicitly by Weiss in [We00]. Since finitely generated linear groups (i.e., subgroups of $GL_n(F)$ where $F$ is a field) are residually finite (by [Ma40]) they are sofic. It is easy to check that a countable group is sofic if and only if all of its finitely generated subgroups are sofic. Thus all countable linear groups are sofic. It is unknown whether every countable group is sofic but an unresolved case is that of the universal Burnside group on a finite set of generators. Pestov has written a beautiful up-to-date survey [Pe08] on sofic groups and their siblings, hyperlinear groups.

\subsection{Sofic entropy}\label{sec:soficentropy} 

Let $G$ be a countable discrete group. In this paper, an {\em action} of $G$ is a triple $(T,X,\mu)$ where $(X,\mu)$ is a standard probability space and $T=(T_g)_{g\in G}$ is a collection of measure preserving transformations $T_g:X \to X$ such that $T_gT_h=T_{gh}$ for all $g,h \in G$. The notation $G \cc^T (X,\mu)$ means $(T,X,\mu)$ is an action. Also $G \cc (X,\mu)$ means that $G$ acts on $(X,\mu)$ by measure-preserving transformations and the product of $g\in G$ with $x\in X$ is denoted $gx$.

 A {\em process} over $G$ if a quadruple $\bX:=(T,X,\mu,\phi)$ where $(T,X,\mu)$ is a $G$-action and $\phi:X \to A$ is a measurable map into a finite or countable set $A$. $\phi$ is called an {\em observable} and $A$ is the {\em range} of the process. We will implicitly assume that the range of every process considered in this paper is finite. The next few paragraphs define the entropy rate of $\phi$ with respect to a sofic approximation $\Sigma$ for $G$ in the special case in which $A$ is finite.

Suppose that $\sigma:G \to \sym(m)$ is a map and $\psi:\{1,\ldots,m\} \to A$ is a function. In order to compare $\psi$ with $\phi$, let $W \subset G$ be finite ($W$ is for {\em window}). Recall that $A^W$ is the set of all functions from $W$ to $A$. Let $\phi^W:X \to A^W$ be the map defined $\phi^W(x)(w):=\phi(T_wx)$. Similarly, define $\psi^W_\sigma:\{1,\ldots,m\} \to A^W$ by $\psi_\sigma^W(p)(w) = \psi(\sigma(w)p)$. The measure $\mu$ pushes forward to a measure $\phi^W_*\mu$ on $A^W$. Similarly, if $u$ is the uniform probability measure on $\{1,\ldots,m\}$, then $(\psi_\sigma^W)_*u$ is a measure of $A^W$. Let $d_W((\sigma,\psi),\phi)$ be the total variation distance between $\phi^W_*\mu$ and $(\psi_\sigma^W)_*u$. Explicitly,
$$d_W((\sigma,\psi),\phi):=\frac{1}{2}||\phi^W_*\mu -(\psi_\sigma^W)_*u||_1= \frac{1}{2}\sum_{a \in A^W} \Big| \phi^W_*\mu\big(\{a\}\big) - \big(\psi^W_\sigma\big)_*u\big(\{a\}\big)\Big|.$$


The {\em sofic entropy rate} of the process $\bX$ with respect to a sofic approximation $\Sigma=\{\sigma_i\}_{i=1}^\infty$ to $G$ (where $\sigma_i:G\to \sym(m_i)$) is defined by:
\begin{eqnarray}\label{eqn:sofic}
h(\Sigma,\bX):=\inf_{W \subset G} \inf_{\epsilon>0} \limsup_{i\to\infty} \frac{\log\#\{\psi:\{1,\ldots,m_i\} \to A:~d_W((\sigma_i,\psi),\phi)<\epsilon\}}{m_i}.
\end{eqnarray}
The first infimum is over all finite subsets of $G$. The entropy $h(\Sigma,\bX)$ may alternatively be denoted by $h(\Sigma,\phi)$ or $h_\mu(\Sigma,\phi)$.

In order to obtain a measure-conjugacy invariant, consider a special class of observables as follows. The map $\phi$ is {\em generating} if the smallest $G$-invariant $\sigma$-algebra on $X$ for which $\phi$ is measurable is the $\sigma$-algebra of all measurable sets up to sets of measure zero. The following is part of the main result of [Bo10a].
\begin{thm}\label{thm:K}
Suppose $G \cc (X,\mu)$. If $\phi_1$ and $\phi_2$ are finite generating observables of $X$ and $\Sigma$ is a sofic approximation to $G$ then $h(\Sigma,\phi_1)=h(\Sigma,\phi_2)$.
\end{thm}
Because of this result, the {\em entropy} of the action $G \cc^T (X,\mu)$ with respect to $\Sigma$ is defined by $h(\Sigma,T):=h(\Sigma,\phi)$ where $\phi$ is any finite generating observable (if one exists). 

 
In [Bo10a] an alternative but equivalent definition of entropy is given based on partitions instead of observables. Also the entropy rate of an observable with countable range is defined under special conditions. That extension is not needed here. The main result of this paper is:
 \begin{thm}\label{thm:main}
If $G$ is infinite and amenable, $G\cc (X,\mu)$ and $\phi$ is a finite observable then for any sofic approximation $\Sigma$ to $G$, $h(\Sigma,\phi)$ is the classical mean entropy rate of $\phi$.
\end{thm}
The definition of classical mean entropy rate is reviewed in \S \ref{sec:amenable}. By [Ro88], if the classical entropy of an ergodic, essentially free system $G \cc^T (X,\mu)$ is finite (and $G$ is amenable) then there exists a finite generating observable for the action. So the above theorem implies $h(\Sigma,T)$ is the classical entropy of the action in this case.

\subsection{Overview}
\S \ref{sec:random} discusses several variations on the definition of sofic entropy. These arise from allowing randomness in the sofic approximation and also in the approximations to the observable. \S \ref{sec:amenable} provides basic background on amenable groups and classical entropy theory. \S \ref{sec:Z} proves the main Theorem \ref{thm:main} in the special case in which $G=\Z$. This uses the above variations on sofic entropy but is otherwise elementary. \S \ref{sec:relative} discusses relative entropy theory; both the classical case and the sofic case. \S \ref{sec:oe} proves the main results in entropy/orbit-equivalence theory that allow us to conclude Theorem \ref{thm:main}. 

{\bf Acknowledgements}. I'd like to thank Gabor Elek for providing a rough outline of a proof of the main theorem based on quasi-tiling machinery. The proof presented here does not use his outline but it helped get me started. I'd also like to thank David Kerr for encouragement and especially Hanfeng Li for finding many errors in previous versions.

\section{Random sofic approximations, approximate processes and entropy}\label{sec:random} 
It will be helpful to broaden the notion of sofic approximation to allow for `random' sofic approximations, defined next.

\begin{defn}[Random sofic approximation]\label{defn:random}
Let $G$ be a countable group. Let $\{m_i\}_{i=1}^\infty$ be a sequence of natural numbers, $\sym(m_i)^G$ the set of maps from $G$ to $\sym(m_i)$ and $u_{m_i}$ the uniform probability measure on $[m_i]$. Let $\kappa_i$ be a probability measure on $\sym(m_i)^G$. We say that the sequence $\sK=\{\kappa_i\}_{i=1}^\infty$ is a {\em random sofic approximation} to $G$ if for every $g,h \in G$,
$$\lim_{i\to\infty}  \kappa_i\times u_{m_i}\left(\{(\sigma,p) \in \sym(m_i)^G \times [m_i]:~\sigma(g)\sigma(h)p = \sigma(gh)p\}\right) = 1$$
and for every $g\ne h \in G$,
$$\lim_{i\to\infty} \kappa_i\times u_{m_i}\left(\{(\sigma,p) \in \sym(m_i)^G \times [m_i]:~\sigma(g)p \ne \sigma(h)p\}\right) = 1.$$
\end{defn}

\begin{remark}\label{rem:1}
This notion generalizes sofic approximations in the following sense. If $\Sigma=\{\sigma_i\}_{i=1}^\infty$ is a sofic approximation of a group $G$ (where $\sigma_i \in \sym(m_i)^G$) and $\delta_i$ is the probability measure on $\sym(m_i)^G$ supported on $\sigma_i$ then $\{\delta_i\}_{i=1}^\infty$ is a random sofic approximation of $G$. 
\end{remark}


\begin{defn}[Sofic entropy]
Let $\bX:=(T,X,\mu,\phi)$ be a process over a group $G$ with random sofic approximation $\sK=\{\kappa_i\}_{i=1}^\infty$. Suppose $\phi:X \to A$ where $A$ is a finite set. For each $\sigma \in \sym(m_i)^G$, finite $W \subset G$ and $\epsilon>0$, let $N_i(\sigma,W,\epsilon)$ be the number of functions $\psi:[m_i] \to A$ such that $d_W((\sigma,\psi),\phi)<\epsilon$ (this is the notation used in \S \ref{sec:soficentropy}). The {\em sofic entropy} of $\bX$ with respect to $\sK$ is:
$$h(\sK,\bX):=\inf_{W\subset G}\inf_{\epsilon>0} \limsup_{i\to\infty} \frac{ \int \log N_i(\sigma,W,\epsilon)~d\kappa_i(\sigma)}{m_i}.$$
For example, if $\sK'$ is any subsequence of $\sK$ then $h(\sK',\bX) \le h(\sK,\bX)$.
\end{defn}

The definition above generalizes the notion of sofic entropy by introducing randomness into the sofic approximation. It is also possible to introduce randomness into the observables. This leads to a new notion of entropy called upper-sofic entropy (which was introduced implicitly in [Bo10b]). To explain, it is necessary to have a notion of ``approximate process'' which is motivated by the definition of a sofic group.
\begin{defn}[Approximate process]\label{defn:process}
An {\em approximate process} over $G$ is a quadruple $\bX=(T,X,\mu,\phi)$ where $(X,\mu)$ is a standard probability space, $T=(T_g)_{g\in G}$ is a set of measure-preserving Borel maps $T_g:X \to X$ and $\phi:X \to A$ is a Borel map to a finite or countable set $A$ called the {\em range} of the process. An approximate process is a {\em process} if $T$ defines an action: i.e., $T_{gh}=T_gT_h$ for all $g,h \in G$. The word `approximate' is used to suggest that $\bX$ is approximating some process. The definition by itself does not imply this but it is how these objects will be used.
\end{defn}

\begin{defn}[Local statistics and distance between processes]\label{defn:localstatistics}
Given a finite set $W \subset G$ and an approximate process $\bX=(T,X,\mu,\phi)$ define $\phi_T^W:X \to A^W$ by 
$$\phi_T^W(x):=\big[ w \mapsto \phi(T_w x)\big].$$
Let $(\phi_T^W)_*\mu$ be the pushforward measure on $A^W$. This measure is called the {\em $W$-local statistics of $\bX$}.

Given another approximate process $\bY=(S,Y,\nu,\psi)$ with range $A$ we define
$$d_W(\bX,\bY):= \frac{1}{2}\| (\phi_T^W)_*\mu - (\psi_S^W)_*\nu \|_1.$$
$\bX$ and $\bY$ are said to be equivalent if $d_W(\bX,\bY)=0$ for all finite $W \subset G$. Let $\cP(G,A)$ be the set of all equivalence classes of approximate processes over $G$ with range $A$. By abuse of notation, we do not distinguish between an approximate process and its equivalence class. Let $\cP(G,A)$ have the smallest topology such that for every finite $W\subset G$ the function $d_W$ is continuous with respect to the product topology on $\cP(G,A)\times \cP(G,A)$. 
\end{defn}

\begin{defn}[Approximate processes from random sofic approximations]\label{defn:eta}
Let $\sK=\{\kappa_i\}_{i=1}^\infty$ be a random sofic approximation to $G$. Let $\eta_i$ be a probability measure on $\sym(m_i)^G\times A^{[m_i]}$ where $A$ is a finite or countable set. Suppose that the projection of $\eta_i$ to the first factor is $\kappa_i$. For each $g \in G$ define
$$\hsigma_i(g):\sym(m_i)^G \times A^{[m_i]}\times  [m_i] \to \sym(m_i)^G \times A^{[m_i]}\times [m_i], ~~\hsigma_i(g)(\sigma,\psi,p) = (\sigma, \psi,\sigma(g)p).$$
Define 
$$\chi_i:\sym(m_i)^G \times A^{[m_i]}\times  [m_i] \to A,~~ \chi_i(\sigma,\psi,p):= \psi(p).$$
Define
$$\bX_i:=(\hsigma_i, \sym(m_i)^G \times A^{[m_i]} \times [m_i], \eta_i \times u_{m_i}, \chi_i).$$
Then $\bX_i$ is the approximate process {\em constructed from $\eta_i$}. 

The sequence $\{\bX_i\}_{i=1}^\infty$ is {\em adapted to the sofic approximation $\sK=\{\kappa_i\}_{i=1}^\infty$} if it arises from the above construction for some sequence of measures $\{\eta_i\}_{i=1}^\infty$.
\end{defn}
\begin{defn}\label{defn:Shannon}
If $\mu$ is a probability measure on a finite or countable set $X$, then
$$H(\mu) := - \sum_{x\in X} \mu(\{x\})\log(\mu(\{x\})).$$
By convention $0\log(0)=0$.
\end{defn}

\begin{defn}[Upper-sofic entropy]\label{defn:upper}
Let $\eta_i, \kappa_i$, etc. be as in definition \ref{defn:eta} and let $\eta_i = \int \nu_{i,\sigma} ~d\kappa_i(\sigma)$ be the decomposition over $\kappa_i$. So $\nu_{i,\sigma}$ is a probability measure on the set $\{(\sigma, \xi) \in \sym(m_i)^G \times A^{[m_i]}\}$. Then define
$$h(\bX_i) := \frac{1}{m_i} \int H(\nu_{i,\sigma}) ~d\kappa_i(\sigma).$$
This definition depends implicitly on $\eta_i$ (which might not by determined by the equivalence class of $\bX_i$).

The {\em upper-sofic entropy} of a finite-range process $\bX=(T,X,\mu,\phi)$ with respect to $\sK$  is defined by
$$\bh(\sK,\bX):=\sup \limsup_{j\to\infty} h(\bX_{j})$$
where the supremum is over all sequences $\{\bX_j\}_{j=1}^\infty$ of approximate processes adapted to $\sK'$ (where $\sK'$ is a subsequence of $\sK$) such that $\lim_{j\to\infty} \bX_{j} = \bX$. If no such exists then $\bh(\sK,\bX)=-\infty$. The upper-sofic entropy $h(\Sigma,\bX)$ can alternatively be denoted by $\bh(\sK,\phi)$ or $\bh_\mu(\sK,\phi)$ if it is desirable to emphasize the dependence on $\mu$ or $\phi$.

If each $\kappa_i$ is supported on a singleton set $\{\sigma_i\} \subset \sym(m_i)^G$ then let $\Sigma:=\{\sigma_i\}_{i\in\N}$ and define $\bh(\Sigma,\bX):=\bh(\sK,\bX)$.
\end{defn}

Using the methods of [Bo10a], it can be shown that upper-sofic entropy is an invariant: if $\phi, \psi$ are two generating observables with finite range then $\bh(\sK,\phi)=\bh(\sK,\psi)$ (but this is not needed here). Next upper-sofic entropy is related to sofic entropy (which will be referred to as lower-sofic entropy so as to distinguish it).
\begin{defn}[Strong convergence of approximate processes]
As above, let $\{\bX_i\}_{i=1}^\infty$ be a sequence of approximate processes constructed from measures $\{\eta_i\}_{i=1}^\infty$ on $\sym(m_i)^G \times A^{[m_i]}$ as in definition \ref{defn:eta}. Suppose that the limit $\lim_{i\to\infty} \bX_i=\bX$ is a process over $G$. The sequence $\{\bX_i\}_{i=1}^\infty$ {\em converges to $\bX$ strongly} (denoted $\lim_{i\to\infty} \bX_i = \bX$ strongly) if for every finite $W \subset G$ and every $\epsilon>0$
$$\lim_{i\to\infty} \eta_i\Big(\big\{(\sigma,\psi)\in \sym(m_i)^G\times A^{[m_i]}:~d_W((\sigma,\psi),\phi)<\epsilon \big\}\Big) =1.$$
\end{defn}

\begin{prop}
Let $\bX$ be a process over $G$ and let $\sK$ be a random sofic approximation. Then 
$$h(\sK,\bX) = \sup \limsup_{j\to\infty} h(\bX_{j})$$
where the supremum is over all sequences $\{\bX_j\}_{j=1}^\infty$ of approximate processes adapted to $\sK'$ (where $\sK'$ is a subsequence of $\sK$) such that $\lim_{j\to\infty} \bX_{j} = \bX$ strongly.
\end{prop}
\begin{proof}
The proof is an exercise in understanding the definitions.
\end{proof}

\begin{cor}\label{cor:bhinequality}
Let $\bX$ be a process over $G$ and let $\sK$ be a random sofic approximation. Then 
$$h(\sK,\bX) \le \bh(\sK,\bX).$$
\end{cor}

\section{Amenable groups}\label{sec:amenable} 

\begin{defn}
Let $G$ be a countable group, $F, K \subset G$ finite sets and $\epsilon>0$. Then $F$ is {\bf left-$(K,\epsilon)$-invariant} if 
$$\frac{|KF \cap F|}{|F|} \ge 1-\epsilon$$
where $KF=\{kf \in G~|~k\in K, f\in F\}$.
A {\bf left-F\o lner sequence} of $G$ is a sequence $\{F_n\}_{n=1}^\infty$ of finite subsets of $G$ such that for all finite $K \subset G$ and all $\epsilon>0$ there exists an $N$ such that $n>N$ implies $F_n$ is left-$(K,\epsilon)$-invariant. $G$ is {\bf amenable} if there exists a left-F\o lner sequence for $G$.
\end{defn}


\begin{defn}
Let $G$ be an amenable group with left-F\o lner sequence $\{F_n\}_{n=1}^\infty$. Let $\bX=(T,X,\mu,\phi)$ be a process over $G$ with range $A$. For a finite $W \subset G$, let $\phi^W:X \to A^W$ be the map $\phi^W(x)(w)=\phi(T_wx)$. The {\em classical entropy} of $\bX$ is defined by
$$h(\bX):=\lim_{n\to\infty} \frac{H(\phi^{F_n}_*\mu)}{|F_n|}$$
where $H(\cdot)$ is as in definition \ref{defn:Shannon}. Some alternative notation for the entropy rate are: $h(\bX)=h(\phi)=h_\mu(\phi)=h_\mu(T,\phi)=h(T,\phi)$. The entropy rate does not depend on the choice of F\o lner sequence (e.g., [Ol85]). 
\end{defn}

\section{The case of $\Z$}\label{sec:Z}     
The purpose of this section is to prove:
\begin{prop}\label{prop:Z}
Let $\sK=\{\kappa_i\}_{i=1}^\infty$ be a random sofic approximation of $\Z$. Let $\bX$ be a process over $\Z$ with finite range. Then $\bh(\sK,\bX) = h(\sK,\bX)= h(\bX)$. That is, classical entropy, sofic entropy and upper-sofic entropy agree.
\end{prop}

In order to prove this, we will reduce to the case when the sofic approximation $\sK$ is particularly simple (given by finite quotients of $\Z$). For this, we need to define what it means for two sofic approximations to be close.
\begin{defn}\label{defn:close}
Let $W \subset G$ be finite and let $\epsilon>0$. Suppose that $\sigma:G\to\sym(m)$ and $\sigma':G \to \sym(m')$ are two maps and there exist subsets $Q \subset [m]$, $Q' \subset [m']$ and a bijection $\beta:Q \to Q'$ such that
\begin{enumerate}
\item $\sigma'(w)\beta(q) = \beta(\sigma(w)q)$ for all $w \in W$ and $q\in Q$ with $\sigma(w)q \in Q$;
\item $\sigma(w)\beta^{-1}(q') = \beta^{-1}(\sigma'(w)q')$ for all $w \in W$ and $q'\in Q'$ with $\sigma'(w)q' \in Q'$;
\item $|Q| \ge (1-\epsilon)m$, $|Q'| \ge (1-\epsilon)m'$.
\end{enumerate}
Then $\sigma$ and $\sigma'$ are said to be {\em $(W,\epsilon)$-close} to each other.  We say two probability measures $\kappa, \kappa'$ on $\sym(m)^G, \sym(m')^G$ respectively are {\em $(W,\epsilon)$-close} if there is a probability measure $\vartheta$ on $\sym(m_i)^G \times \sym(m_i')^G$ with marginals $\kappa$ and $\kappa'$ such that $\vartheta(\cG(W,\epsilon)) \ge 1-\epsilon$ where  $\cG(W,\epsilon)$ is the set of all $(\sigma,\sigma') \in \sym(m)^G \times \sym(m')^G$ that are $(W,\epsilon)$-close to each other. Finally, we say that two random sofic approximations to $G$, $\sK=\{\kappa_i\}_{i=1}^\infty$ and $\sL=\{\lambda_i\}_{i=1}^\infty$ are {\em asymptotic} if for every finite $W\subset G$ and $\epsilon>0$, $\kappa_i$ is $(W,\epsilon)$-close to $\lambda_i$ for all sufficiently large $i$. 
 \end{defn}



The following theorem is a special case of Theorem \ref{thm:asymptotic2r} proven in the next section.
\begin{thm}\label{thm:asymptotic2}
Let $\bX$ be a process over a group $G$ with random sofic approximations $\sK$ and $\sL$. If $\sK$ and $\sL$ are asymptotic then $\bh(\sK,\bX)=\bh(\sL,\bX)$ and $h(\sK,\bX) = h(\sL,\bX)$.
\end{thm}

This motivates the next result, which is a minor extension of \cite{ES10}, Theorem 2.
\begin{thm}\label{thm:Zasymptotic}
Let $G$ be an amenable group and $\sK=\{\kappa_i\}_{i=1}^\infty, \sL=\{\lambda_i\}_{i=1}^\infty$ be random sofic approximations to $G$. Suppose that for each $i$, $\kappa_i:G \to \sym(m_i)$, $\lambda_i:G \to \sym(m'_i)$ and $\lim_{i\to\infty} \frac{m_i}{m'_i}=1$. Then $\sK$ and $\sL$ are asymptotic.
\end{thm}

To prove this, we need some terminology adapted from \cite{ES10}. 


\begin{defn}
Let $G$ be a finitely generated group and $S=S^{-1} \subset G$ a finite generating set for $G$. Let $\sigma:G \to \sym(m)$ be a map. Then $\sigma$ is an {\em $r$-approximation} to $(G,S)$ (for $r$ a positive integer) if there exists a set $V \subset [m]$ such that $|V| \ge (1-\frac{1}{r})m$ and for every $v\in V$ and every sequence $g_1,\ldots, g_r \in S\cup \{e\}$,
$$\sigma(g_1)\sigma(g_2)\cdots\sigma(g_r)v = \sigma(g_1g_2\cdots g_r)v$$
and if $ g_1',\ldots, g_r' \in S\cup\{e\}$ are such that $g_1\cdots g_r \ne g_1'\cdots g_r'$ then
$$\sigma(g_1g_2 \cdots g_r)v\ne \sigma(g_1'g_2'\cdots g_r')v.$$


\end{defn}

\begin{defn}
If $T \subset G$ and $\sigma:G \to \sym(T)$ is a map then $\sigma$ is a {\em copy} of $T$ if $\sigma(g)t=gt$ for every $t\in T$ and $g\in G$ with $gt \in T$. 
\end{defn}



\begin{defn}
If $\sigma_i: G\to \sym(m_i)$ ($i=1,\ldots,n$) are maps then $\sigma_1+\cdots + \sigma_n$ is defined to be the map from $G$ to $\sym(m_1+\cdots + m_n)$ defined by
$$(\sigma_1+\cdots + \sigma_n)(g)p:=\sigma_i(g)p$$
where $i$ is determined by: if $p\le m_1$ then $i=1$; otherwise $i$ is such that $\sum_{j=1}^{i-1} m_{j} < p \le \sum_{j=1}^{i} m_{j}$. Also if $r\ge 1$ is an integer then $r\sigma_1:=\sigma_1+\cdots + \sigma_1$ ($r$ summands). So if $\alpha=(\alpha_1,\ldots,\alpha_r)$ is a vector of positive integers and $(\sigma_1,\ldots,\sigma_r)$ is a vector of maps then $\alpha \cdot (\sigma_1,\ldots,\sigma_r)$ is defined to be the map $\alpha_1\sigma_1+\cdots+ \alpha_r \sigma_r$.
\end{defn}

\begin{prop}
Let $G$ be a finitely generated amenable group with finite generating set $S=S^{-1}$. Let $\{F_i\}_{i=1}^\infty$ be a left-F\o lner sequence for $G$. Then for each integer $r>0$ there exists an integer $R(r)>0$ (which also depends on $\{F_i\}_{i=1}^\infty$), a finite subsequence $\{T_1,\ldots,T_n\} \subset \{F_i\}_{i=1}^\infty$ and a vector $\alpha=(\alpha_1,\ldots,\alpha_n)$ of positive natural numbers such that the following holds. Every $T_i$ satisfies $|ST_i \Delta T_i| \le |T_i|/r$ and every map $\sigma:G \to \sym(m)$ which is an $R(r)$-approximation to $(G,S)$ is  $(S,\frac{1}{r})$-close to an integer multiple of $\alpha \cdot (\sigma_1,\ldots, \sigma_n)$ where $\sigma_i$ is any copy of $T_i$.
\end{prop}
\begin{proof}
This is Proposition 2.8 of \cite{ES10} (in different terminology).
\end{proof}


\begin{proof}[Proof of Theorem \ref{thm:Zasymptotic}]

Let $W \subset G$ be finite and $\epsilon>0$. It suffices to show that $\kappa_i$ is $(W,\epsilon)$-close to $\lambda_i$ for all sufficiently large $i$. Let $S=W \cup W^{-1}$. Without loss of generality, we may assume $G$ is generated by $S$.

Given integers $m,r>0$, let $X(m,r)$ be the set of all $(\sigma,p) \in \sym(m)^G$ such that
$$\sigma(g_1)\sigma(g_2)\cdots \sigma(g_r)p=\sigma(g_1\cdots g_r)p\quad \forall g_1,g_2,\ldots, g_r \in S\cup\{e\}$$
and if $g_1',\ldots, g_r' \in S\cup\{e\}$ are such that $g_1g_2\cdots g_r \ne g_1'g_2'\cdots g_r'$ then
$$\sigma(g_1g_2 \cdots g_r)p\ne \sigma(g_1'g_2'\cdots g_r')p.$$
For $\sigma \in \sym(m)^G$, let $X_\sigma(m,r)$ be the set of all $p \in [m]$ with $(\sigma,p) \in X(m,r)$.

Now let $\sK=\{\kappa_i\}_{i=1}^\infty$ be a random sofic approximation. By definition, there exists an $N$ such that for all $i>N$, $\kappa_i\times u_{m_i} (X(m_i,r)) \ge 1-\frac{1}{r^2}$ where $u_{m_i}$ is the uniform probability measure on $[m_i]$. Let $\sigma$ be an element of $\sym(m_i)^G$ chosen at  random with law $\kappa_i$. By Markov's inequality,
$$\mathbb{P}( m_i-|X_\sigma(m_i,r)| > m_i/r ) \le \bE[ m_i-|X_\sigma(m_i,r)| ] (r/m_i) \le (m_i/r^2)(r/m_i) = (1/r).$$
Therefore, 
$$\kappa_i( \{\sigma \in \sym(m_i)^G:~|X_\sigma(m_i,r)| \ge m_i(1-1/r)\}) \ge 1-(1/r).$$
Note that each $\sigma \in \sym(m_i)^G$ with $|X_\sigma(m_i,r)| \ge m_i(1-1/r)$ is an $r$-approximation to $(G,S)$.

By the previous Proposition, there exist a finite sequence $\{T_1,\ldots,T_n\}$ of finite sets $T_j \subset G$ and a vector $\alpha=(\alpha_1,\ldots,\alpha_n)$ of positive natural numbers such that the following holds. Every $T_j$ satisfies $|ST_j \Delta T_j| \le |T_j|/r$ and for every map $\sigma:G \to \sym(m_i)$ which is an $r$-approximation to $(G,S)$  there is an integer $k_i(\sigma)$ such that $\sigma$ is $(S,\frac{1}{r})$-close  to $k_i(\sigma)\alpha \cdot (\sigma_1,\ldots, \sigma_n)$ where $\sigma_j$ is a copy of $T_j$ for $1\le j \le n$.

We would like to choose the integers $k_i(\sigma)$ to be independent of $\sigma$ (although they must depend on $i$). So let $k_i$ be the minimum of $k_i(\sigma)$ over all maps $\sigma:G\to \sym(m_i)$ which are $r$-approximations to $(G,S)$. We claim that each such $\sigma$ is $(S,\frac{3}{r})$-close to $k_i\alpha \cdot (\sigma_1,\ldots, \sigma_n)$. 

To prove the claim, first note that if $\sigma: G\to \sym(m)$ and $\sigma':G\to \sym(m')$ are any two maps that are $(S,\frac{1}{r})$-close to each other then $1-1/r \le \frac{m}{m'}  \le \frac{1}{1-1/r}$. So let $q$ be the integer such that $\alpha \cdot (\sigma_1,\ldots, \sigma_n)$ is a map from $G$ into $\sym(q)$. Then if $\sigma$ is any map from $G$ into $\sym(m_i)$ which is an $r$-approximation to $(G,S)$ then $1-1/r \le \frac{k_i(\sigma)q}{m_i} \le \frac{1}{1-1/r}$. Since this is true for every such $\sigma$, it follows that for any such $\sigma$,
$$1-2/r \le (1-1/r)^2 \le \frac{k_iq}{k_i(\sigma)q}.$$
Clearly, $k_i(\sigma)\alpha \cdot (\sigma_1,\ldots, \sigma_n)$ and $k_i\alpha \cdot (\sigma_1,\ldots, \sigma_n)$ are $(S,\epsilon)$-close for any $\epsilon$ with $1-\epsilon \le \frac{k_iq}{k_i(\sigma)q}$. So we can choose $\epsilon=2/r$. 

It follows readily from the definition of closeness that if $\sigma$ is $(S,\epsilon)$-close to $\sigma'$, $\sigma'$ is $(S,\epsilon')$-close to $\sigma''$ and $\epsilon+\epsilon'\le 1$ then $\sigma$ is $(S,\epsilon+\epsilon')$-close to $\sigma''$. So: if $\sigma:G \to \sym(m_i)$ is an $r$-approximation to $(G,S)$ then it is $(S,1/r)$-close to $k_i(\sigma)\alpha \cdot (\sigma_1,\ldots, \sigma_n)$ which is $(S,2/r)$-close to $k_i\alpha \cdot (\sigma_1,\ldots, \sigma_n)$. Thus $\sigma$ is $(S,3/r)$-close to $k_i\alpha \cdot (\sigma_1,\ldots, \sigma_n)$ (whenever $r \ge 3$).

Let $\tau_i = k_i\alpha \cdot (\sigma_1,\ldots, \sigma_n)$. From the claim it follows that $\kappa_i$ is $(S,3/r)$-close to the probability measure supported on $\{\tau_i\}$.

By choosing $N$ larger if necessary, we may assume that for $i>N$, there is (by a similar argument) an integer $k'$ such that if $\tau'_i=k_i\alpha \cdot (\sigma_1,\ldots, \sigma_n)$ then $\lambda_i$ is $(S,\frac{3}{r})$-close to the probability measure supported on $\{\tau'_i\}$. Note $1-3/r\le \frac{k'_iq}{m'_i} \le \frac{1}{1-3/r}$. So 
$$(1-3/r)^2\min(m_i/m'_i, m'_i/m_i) \le \frac{k'_iq}{k_iq} \le \frac{\max(m_i/m'_i, m'_i/m_i)}{(1-3/r)^2}.$$
This implies $\tau_i$ and $\tau'_i$ are $(S, 1 - (1-3/r)^2\min(m_i/m'_i, m'_i/m_i))$-close. So $\kappa_i$ and $\lambda_i$ are $(S, \frac{6}{r} + 1 - (1-3/r)^2\min(m_i/m'_i, m'_i/m_i) )$-close to each other.  Since $\lim_{i\to\infty} \frac{m_i}{m'_i} = 1$ and $r$ can be made arbitrarily large, we have shown that for every $\epsilon>0$ there exists an $N$ such that $i>N$ implies $\kappa_i$ and $\lambda_i$ are $(S,\epsilon)$-close which implies the theorem.


\end{proof}


\begin{lem}
Let $\sK=\{\kappa_i\}_{i=1}^\infty$ be a random sofic approximation of $\Z$. Let $\bX=(T,X,\mu,\phi)$ be a process over $\Z$ with finite range $A$. Then $h(\sK,\bX) = \bh(\sK,\bX).$
\end{lem}

\begin{proof}
By Theorem \ref{thm:Zasymptotic} and Theorem \ref{thm:asymptotic2}, it suffices to show $h(\Sigma,\phi) = \bh(\Sigma,\phi)$ where $\Sigma=\{\sigma_i\}_{i=1}^\infty$ and $\sigma_i:\Z\to \sym(m_i)$ is the homomorphism with $\sigma_i(1)=(1,2,\ldots,m_i)$. 

By Corollary \ref{cor:bhinequality}, it suffices to show that $h(\Sigma,\phi) \ge \bh(\Sigma,\phi)$. Let $\epsilon>0$ and $W \subset G$ be finite. It suffices to show that
$$\liminf_{n\to\infty} \frac{\log\big( \#\{\psi:\{1,\ldots,m_n\} \to A:~d_W((\sigma_n,\psi),\phi)<\epsilon\}\big)}{m_n} \ge \bh(\Sigma,\phi) - \epsilon.$$

The definition of $\bh(\Sigma,\phi)$ implies there exists an $i$ and a probability measure $\nu_i$ on $A^{[m_i]}$ such that if $\delta_i$ is the probability measure supported on the singleton $\{\sigma_i\}$ and $\bX_i$ is the approximate process constructed from $\eta_i:=\delta_i\times \nu_i$  then
\begin{itemize}
\item $d_W(\bX_i,\bX)< \epsilon/2$;
\item $h(\bX_i) \ge \bh(\Sigma,\phi) - \epsilon/2$;
\item there is an $N>0$ such that $W \subset [-N,N]$ and $\frac{2N}{m_i} < \epsilon/2$. 
\end{itemize}
After perturbing $\nu_i$ if necessary, we may assume that there is an integer $d>0$ such that $d\nu_i$ is integral (i.e., $d\nu_i(\{\xi\}) \in \Z~\forall \xi \in A^{[m_i]}$). We may also assume that $\nu_i$ is $\sigma_i$-invariant by replacing it with $\frac{1}{m_i}\sum_{j=1}^{m_i} (\sigma^j_i)_*\nu_i$ if necessary.

Let $n>i$ be a large number (to be specified later). Let $k=\lfloor \frac{m_n}{dm_i} \rfloor$. We will say that a function $\psi:[m_n] \to A$ is {\em good} if for every $\xi \in A^{[m_i]}$, the number of $j$ with $0 \le j <kd$ satisfying
$$\psi(jm_i + p) = \xi(p), ~~\forall 1\le p \le m_i$$
is exactly $kd\nu_i(\{\xi\})$. For such a $\psi$ let $\bX_\psi$ be the approximate process constructed from $\delta_n \times \delta_\psi$ where $\delta_\psi$ is the probability measure concentrated on $\{\psi\} \subset A^{[m_n]}$. 

In order to estimate $d_W(\bX_{\psi}, \bX_i)$, let $u_{m_n}$ be the uniform probability measure on $[m_n]$ (so $\psi^W_*u_{m_n}$ is the $W$-local statistics of $\bX_{\psi}$ by definition \ref{defn:localstatistics}). If $u'_{m_n}$ is the uniform probability measure on 
$$K=\{p \in [m_n]:~ p=jm_i + q \textrm{ for some } 0\le j <kd \textrm{ and some } 1 + N\le q\le m_i-N\}$$
then $\psi^W_*u'_{m_n}$ is the $W$-local statistics of $\bX_i$. This uses the fact that $\nu_i$ is $\sigma_i$-invariant and $W \subset [-N,N]$. Since $|K| = kd(m_i-2N)$, it follows that
\begin{eqnarray*}
d_W(\bX_{\psi}, \bX_i)  &\le&  1 - \frac{kd(m_i-2N)}{m_n}.
\end{eqnarray*}
Let $N_0$ be large enough so that if $n>N_0$ then  $1-\frac{kd(m_i-2N)}{m_n} < \epsilon/2$. By choice of $\bX_i$, this implies $d_W(\bX_{\psi}, \bX) < \epsilon$. We will now assume that $n>N_0$.

The number of good functions $\psi: [m_n] \to A$ is
$$|A|^{(m_n-kdm_i)}(kd)! \Big(\prod_{\xi \in A^{[m_i]}} \big(kd\nu_i(\{\xi\})\big)! \Big)^{-1}.$$
Stirling's formula implies that 
$$ \lim_{m_n\to\infty} \frac{\log\Big[|A|^{(m_n-kdm_i)} (kd)! \Big(\prod_{\xi \in A^{[m_i]}} \big(kd\nu_i(\{\xi\})\big)! \Big)^{-1}\Big]}{kd} = H(\nu_i).$$
Therefore,
$$\liminf_{n \to \infty} \frac{\log\big( \#\{\psi:\{1,\ldots,m_n\} \to A:~d_W((\sigma_n,\psi),\phi)<\epsilon\}\big)}{m_n} \ge \frac{H(\nu_i)}{m_i} = h(\bX_i) \ge \bh(\Sigma,\phi) -\epsilon.$$
Because $\epsilon>0$ is arbitrary, this implies the lemma.

\end{proof}

\begin{lem}
Let $\sK=\{\kappa_i\}_{i=1}^\infty$ be a random sofic approximation of $\Z$. Let $\bX=(T,X,\mu,\phi)$ be a process over $\Z$. Then $h(\sK,\bX)\ge h(\bX)$.
\end{lem}

\begin{proof}
By Theorems \ref{thm:Zasymptotic} and \ref{thm:asymptotic2} and the previous lemma it suffices to show that $\bh(\Sigma,\phi) \ge h(\bX)$ where $\Sigma=\{\sigma_i\}_{i=1}^\infty$ and $\sigma_i:\Z\to \sym(m_i)$ is the homomorphism with $\sigma_i(1)=(1,2,\ldots,m_i)$. 

Let $\phi^{m_i}:X \to A^{[m_i]}$ be the map $\phi^{m_i}(x)(p) = \phi(T^px)~\forall x\in X, p \in [m_i]$. Let $\eta_i :=\phi^{m_i}_*\mu$ be the pushforward measure on $A^{[m_i]}$. Let $\bX_i$ be the approximate process constructed from $\delta_i\times \eta_i$ where $\delta_i$ is the probability measure concentrated on $\{\sigma_i\} \subset \sym(m_i)^\Z$. The F\o lner property of the sequence of intervals $[m_i] \subset \Z$ implies that $\lim_{i\to\infty} \bX_i = \bX$. The definition of $h(\bX)$ implies $\lim_{i\to\infty} h(\bX_i) = h(\bX)$. This implies the lemma. 
 \end{proof}

\begin{proof}[Proof of Proposition \ref{prop:Z}]
By Theorems \ref{thm:Zasymptotic}, \ref{thm:asymptotic2} and the previous lemma, it suffices to prove that if $\Sigma=\{\sigma_i\}_{i=1}^\infty$ where $\sigma_i:\Z\to \sym(m_i)$ is the homomorphism with $\sigma_i(1)=(1,2,\ldots,m_i)$ then $\bh(\Sigma,\bX) \le h(\bX)$. Let $\{\nu_i\}_{i=1}^\infty$ be a sequence of probability measures on $A^{[m_i]}$ such that if $\{\bX_i\}_{i=1}^\infty$ is the sequence of approximate processes constructed from $\eta_i:=\delta_i\times \nu_i$ (where $\delta_i$ is the probability measure concentrated on $\{\sigma_i\}_{i=1}^\infty \subset \sym(m_i)^\Z$) then 
$$\lim_{i\to\infty} \bX_i = \bX, ~\textrm{ and }  \lim_{i\to\infty} h(\bX_i) = \bh(\Sigma,\bX).$$
Let $\bX=(T,X,\mu,\phi)$ where $\phi:X\to A$. Using a standard trick, $\bX$ is equivalent to a process of the form $(\tau,A^\Z,\mu',\phi')$ where $\tau:A^\Z \to A^\Z$ is the shift map $\tau(y)(n):=y(n+1)$ and $\phi':A^\Z \to A$ is the time-$0$ projection $\phi'(y):=y(0)$. To be precise, 
let $\phi^\Z:X \to A^\Z$ be the map $\phi^\Z(x)(n):=\phi(T^nx)$ for $x\in X, n\in \Z$. This map is equivariant. $\bX$ is equivalent to $(\tau,A^\Z, \mu',\phi')$ where $\mu'$ is the pushforward measure $\phi^\Z_*\mu$. So without loss of generality, we will assume that $X=A^\Z$, $T=\tau$ is the shift map and $\phi:A^\Z \to A$ is the time-$0$ projection.

 Let $\pi_n:(A^{[m_n]})^\Z \to A^\Z$ be the map defined by
$$\pi_n[\Psi]( im_n + j) := \Psi(i)(j),~~\forall \Psi \in (A^{[m_n]})^\Z, i \in \Z, j \in [m_n].$$
Let $\nu_n^\Z$ be the measure on $(A^{[m_n]})^\Z$ equal to the product of $\Z$-copies of $\nu_n$. Let $\mu'_n = (\pi_n)_*(\nu_n^\Z)$ be the pushforward measure on $A^\Z$. 
Let
$$\mu_n = \frac{1}{m_n} \sum_{j =1}^{m_n} \tau^j_*\mu'_n.$$
Note that $\mu_n$ is $\tau$-invariant. 

For an interval $[a,b] \subset \Z$, let $\phi^{[a,b]}:A^\Z \to A^{[a,b]}$ be the projection map. By concavity of entropy,
\begin{eqnarray*}
h_{\mu_n}(\phi) &=& \lim_{N\to\infty} \frac{H(\phi_*^{[-m_nN+1,m_nN]}\mu_n)}{2m_nN} \ge  \lim_{N\to\infty} \frac{H(\phi_*^{[-m_nN+1,m_nN]}\mu'_n)}{2m_nN}\\
&=& \lim_{N\to\infty} \frac{2N H(\nu_n)}{2m_nN} =  \frac{H(\nu_n)}{m_n}. 
\end{eqnarray*}

Since $\lim_{i\to\infty} \bX_i = \bX$, it follows that $\lim_n \mu_n=\mu$ in the weak* topology on $M(A^\Z)$, the space of all $\tau$-invariant Borel probability measures on $A^\Z$. It is well-known that the function $\lambda \in M(A^\Z) \mapsto h_\lambda(\phi)$ is upper semi-continuous on $M(A^\Z)$. For example, see [Gl03, Lemma 15.1 page 270]. It follows that 
$$\bh(\Sigma,\phi) = \limsup_{n\to\infty} \frac{H(\nu_n)}{m_n} \le \limsup_{n\to\infty} h_{\mu_n}(\phi) \le h_\mu(\phi)$$
as required.
\end{proof}


\section{Relative entropy}\label{sec:relative} 

\begin{defn}[Factors of approximate processes]\label{defn:composition}
Given an approximate process $\bX=(T,X,\mu,\phi)$ with $\phi:X \to A$ and a function $\beta:A \to B$, let $\beta \circ \bX$ be the approximate process
$$\beta \circ \bX:=(T,X,\mu,\beta \circ \phi).$$
If an approximate process $\bY$ is constructed from a measure $\eta_i$ on $\sym(m_i)^G \times A^{[m_i]}$ as in definition \ref{defn:eta} then $\beta \circ \bY$ has an alternative description as follows. Let 
$$\tbeta: \sym(m_i)^G \times A^{[m_i]} \to  \sym(m_i)^G \times B^{[m_i]},~~ \tbeta(\sigma,\xi) = (\sigma, \beta \circ \xi).$$
Then $\beta \circ \bY$ is equivalent to the process $\bZ$ constructed from the pushforward measure $\tbeta_*\eta_i$. (By equivalent, we mean that $d_W(\beta \circ \bY, \bZ) = 0$ for every finite $W \subset G$ in the notation of definition \ref{defn:process}).
\end{defn}

The next lemma follows immediately from the definitions.
\begin{lem}\label{lem:composition}
If $\{\bX_i\}_{i=1}^\infty$ is a sequence of approximate processes with range $A$, $\lim_{i\to\infty} \bX_i = \bX$ and $\beta:A \to B$ is a map then
$\lim_{i\to\infty} \beta \circ \bX_i = \beta \circ \bX$.
Moreover, if $\lim_{i\to\infty} \bX_i = \bX$ strongly then $\lim_{i\to\infty} \beta \circ \bX_i = \beta \circ \bX$ strongly.
\end{lem}

\begin{defn}[Relative entropy]

Let $G$ be a countable amenable group acting by measure-preserving transformations on a standard probability space $(X,\cB,\mu)$. Let $\phi:X \to A$ be a finite observable, $\Sigma$ be a sofic approximation to $G$ and $\cF \subset \cB$ be a $G$-invariant $\sigma$-algebra. Define 
\begin{eqnarray*}
h_\mu(\phi|\cF) &:=& \inf_\psi h_\mu(\phi \vee \psi) - h_\mu(\psi)\\
\end{eqnarray*}
where the infimum is over all finite-range $\cF$-measurable observables $\psi:X \to B$ and $\phi \vee \psi:X \to A\times B$ is the observable $\phi \vee \psi(x)=(\phi(x),\psi(x))$.
In case $\cF$ is the $G$-invariant $\sigma$-algebra generated by an observable $\psi$ then we write $h_\mu(\phi|\psi)=h_\mu(\phi|\cF)$. In case $\bX=(T,X,\mu,\phi)$ is a process over $G$ and $\psi=\beta \circ \phi$ for some $\beta:A \to B$, we write $h(\bX|\beta \circ \bX) = h(\bX)-h(\beta \circ \bX) = h_\mu(\phi|\psi)$. The first equality holds from the Abramov-Rohlin formula [WZ92]. 
\end{defn}

We can now define relative sofic entropy (in a special case).
\begin{defn}[Relative sofic entropy]
Let $\sK$ be a random sofic approximation to $G$. Let $\bX$ be a $G$-process with finite range $A$ and $\beta:A \to B$ a map. Define
$$\bh(\sK,\bX| \beta\circ \bX) = \sup \limsup_{j\to\infty} h(\bX_j) - h(\beta \circ \bX_j)$$
where the supremum is over all sequences $\{\bX_j\}_{j=1}^\infty$ adapted to $\sK'$ (where $\sK'$ is a subsequence of $\sK$) such that $\lim_{j\to\infty} \bX_j = \bX$. Similarly, let 
$$h(\sK,\bX| \beta\circ \bX) = \sup\limsup_{j\to\infty} h(\bX_j) - h(\beta \circ \bX_j)$$
where the supremum is over all sequences $\{\bX_j\}_{j=1}^\infty$ adapted to $\sK'$ (where $\sK'$ is a subsequence of $\sK$) such that $\lim_{j\to\infty} \bX_j = \bX$ strongly. If $\bX=(T,X,\mu,\phi)$ and $\psi = \beta \circ \phi$ then an alternative notation for relative entropy is:
$$\bh(\sK,\phi|\psi) := \bh(\sK,\bX| \beta\circ \bX),~h(\sK,\phi|\psi) := h(\sK,\bX| \beta\circ \bX).$$
We may also write $\bh_\mu(\sK,\phi|\psi)$ or $h_\mu(\sK,\phi|\psi)$ if it is desirable to emphasize the dependence on the measure $\mu$.
\end{defn}

Before moving on, it is worthwhile to record some inequalities relating the entropy of direct products to the entropies of their direct factors. To be precise if $\bX=(T,X,\mu,\phi)$, $\bY=(S,Y,\nu,\psi)$ are two approximate processes over $G$ then their direct product is the process $\bX\times \bY:=(T\times S,X\times Y,\mu\times \nu,\phi\times \psi)$ where 
$$(T\times S)_g(x,y):=(T_gx,S_gy)~\forall g\in G, (x,y)\in X\times Y.$$

Let $\pi_B:A \times B \to B$ and $\pi_A:A \times B \to A$ be the projection maps. To simplify, we let (for example)
$$\bh(\sK, \bX \times \bY| \bY):=\bh(\sK,\bX\times \bY | \pi_B \circ \bX\times \bY).$$

\begin{lem}\label{lem:bhproduct}
If $\bX$, $\bY$ are two processes over $G$ as above and $\sK$ is a random sofic approximation then
\begin{eqnarray*}
\bh(\sK,\bX\times \bY) &\le& \bh(\sK,\bX) + \bh(\sK,\bY)\\
\bh(\sK, \bX \times \bY| \bY) &\le &\bh(\sK,\bX).
\end{eqnarray*}
Similar statements hold with lower-sofic entropy in place of upper-sofic entropy.
\end{lem}

\begin{proof}
Let $\bX=(T,X,\mu,\phi)$ and $\bY=(S,Y,\nu,\psi)$. Let $\{\bZ_i\}_{i=1}^\infty$ be a sequence of approximate processes adapted to $\sK$ so that $\lim_{i\to\infty} \bZ_i=\bX \times \bY$. Then $\{\pi_A\circ \bZ_i\}_{i=1}^\infty$ converges to $\bX$ and $\{\pi_B \circ \bZ_i\}_{i=1}^\infty$ converges to $\bY$. Moreover, 
$$h(\bZ_i) \le h( \pi_A \circ \bZ_i) + h(\pi_B\circ \bZ_i).$$
This and Lemma \ref{lem:composition} imply
\begin{eqnarray*}
\bh(\sK,\bX\times \bY) \le \bh(\sK,\bX) + \bh(\sK,\bY),&& ~\bh(\sK, \bX \times \bY| \bY) \le \bh(\sK,\bX).
\end{eqnarray*}
The proofs for lower-sofic entropy are similar. 
 \end{proof}
If $G$ is non-amenable, then there are examples showing that some of the inequalities of the lemma above can be strict. However, in the special case of Bernoulli actions, we have equality. To be precise, let $(\Omega,\omega)$ be a standard probability space. $G$ acts on the product space $(\Omega^G,\omega^G)$ by $T_gy(f)=y(g^{-1}f)~\forall y \in \Omega^G, g,f\in G$. Let $\phi:\Omega^G \to \Omega$ be the map $\phi(y)=y(e)$. The process $(T,\Omega^G,\omega^G,\phi)$ is the {\em Bernoulli process} over $G$ with base $(\Omega,\omega)$. 

\begin{lem}\label{lem:bhproduct2}
If $\bX$ is any finite-range process over $G$, $\bY$ is a Bernoulli process with base $(\Omega,\omega)$ (where $\Omega$ is finite) and $\sK$ is a random sofic approximation to $G$ then
\begin{eqnarray*}
\bh(\sK,\bX\times \bY) &=& \bh(\sK,\bX) + \bh(\sK,\bY)\\
\bh(\sK, \bX \times \bY| \bY) &= &\bh(\sK,\bX).
\end{eqnarray*}
Similar statements hold with lower-sofic entropy in place of upper-sofic entropy.
\end{lem}

\begin{proof}
The first statement above was proven in [Bo10a] for non-random sofic approximations and with lower-sofic entropy in place of upper-sofic entropy. We will handle here only the case of upper-sofic entropy as the other cases are similar. 

Let $\{\bX_j\}_{j=1}^\infty$ be a sequence of approximate processes adapted to $\sK'=\{\kappa_j\}_{j=1}^\infty$ (where $\sK'$ is a subsequence of $\sK$) such that $\lim_{j\to\infty} \bX_j = \bX$ and $\lim_{j\to\infty} h(\bX_j)=\bh(\sK,\bX)$. Let $\eta_j$ be the probability measure on $\sym(m_j)^G \times A^{[m_j]}$ from which $\bX_j$ is constructed (as in definition \ref{defn:eta}). Let $\teta_j = \eta_j \times \omega^{[m_j]}$ be a probability measure on $\sym(m_j)^G \times A^{[m_j]} \times \Omega^{[m_j]} = \sym(m_j)^G \times (A\times \Omega)^{[m_j]}$ and let $\bZ_j$ be the approximate process constructed from $\teta_j$. Note $h(\bZ_j) = h(\bX_j) + H(\omega)$.

We claim that  $\lim_{j\to\infty} \bZ_j = \bX\times \bY$. To see this, let $V \subset G$ be finite. Let $\cG_j(V)$ be the set of all $(\sigma,p) \in \sym(m_j)^G  \times [m_j]$ such that $\sigma(g)\sigma(h)p = \sigma(gh)p$ for all $g,h \in V$ and for every $g\ne h$ with $g,h\in V$, $\sigma(g)p \ne \sigma(h)p$. We consider $\cG_j(V) \times (A\times \Omega)^{[m_j]}$ as a subset of $\sym(m_j)^G \times (A\times \Omega)^{[m_j]} \times [m_j]$ in the obvious way. Then $\lim_{j\to\infty} \teta_j \times u_{m_j} (\cG_j(V) \times (A\times \Omega)^{[m_j]})=1$ because $\sK$ is a sofic approximation.

In general, if $\tau$ is a measure on a set $J$ and $J_0 \subset J$ then we write $\tau|J_0$ to denote $\tau$ restricted to $J_0$. Let $\chi_j:\sym(m_j)^G \times (A\times \Omega)^{[m_j]} \times [m_j] \to A\times \Omega$ be the projection map $\chi_j(\sigma, \xi,p) = \xi(p)$. Then the measure $(\chi^V_j)_*(\teta_j\times u_{m_j}|\cG_j(V) \times (A\times \Omega)^{[m_j]})$ splits as a product $\tau_j \times \omega^V$ for some measure $\tau_j$ on $A^V$. Let $\bX=(T,X,\mu,\psi)$. Thus,
\begin{eqnarray*}
\lim_{j\to\infty} d_V(\bZ_j,\bX\times \bY) &=& \lim_{j\to\infty} \frac{1}{2} \| (\chi^V_j)_*(\teta_j\times u_{m_j}) - (\psi\times \phi)^V_*\mu\times \omega^G\|_1\\
&=&\lim_{j\to\infty} \frac{1}{2} \| (\chi^V_j)_*(\teta_j\times u_{m_j}|\cG(V)\times (A\times \Omega)^{[m_j]}) - (\psi\times \phi)^V_*\mu\times \omega^G\|_1\\
&=&\lim_{j\to\infty} \frac{1}{2} \| (\chi^V_j)_*(\eta_j\times u_{m_j}|\cG(V) \times A^{[m_j]}) - \psi^V_*\mu\|_1\\
&=&\lim_{j\to\infty}   d_V(\bX_j,\bX) =0.
\end{eqnarray*} 
The first equality holds by definition of $d_V$. The second one holds because $\lim_{j\to\infty} \teta_j\times u_{m_j}(\cG_j(V) \times (A\times \Omega)^{[m_j]})=1$. The third equality holds because $(\chi^V_j)_*(\teta_j\times u_{m_j}|\cG(V)\times (A\times \Omega)^{[m_j]})$ splits as a product $\tau_j \times \omega^V$ and $(\psi\times \phi)^V_*\mu\times \omega^G$ splits as the product $\psi^V_*\mu \times \omega^V$.
 The last equality holds because $\lim_{j\to\infty} \eta_j \times u_{m_j} (\cG_j(V) \times A^{[m_j]})=1$ since $\sK$ is a sofic approximation.

 Since $\lim_{j\to\infty} \bZ_j = \bX\times \bY$, it follows that
\begin{eqnarray*}\label{oneinequality}
\bh(\sK,\bX\times \bY)  \ge \bh(\sK,\bX) +H(\omega), \quad \bh(\sK, \bX \times \bY| \bY) \ge \bh(\sK,\bX).
\end{eqnarray*}
 By the previous lemma, it now suffices to prove that $\bh(\sK,\bY) = H(\omega)$.

Applying the equation above to the case when $\bX$ is trivial, we see that $\bh(\sK,\bY) \ge H(\omega)$. Suppose that $\{\bY_j\}_{j=1}^\infty$ is a sequence of approximate processes adapted to $\sK'=\{\kappa_j\}_{j=1}^\infty$ (where $\sK'$ is a subsequence of $\sK$) such that $\lim_{j\to\infty} \bY_j = \bY$ and $\lim_{j\to\infty} h(\bY_j)=\bh(\sK,\bY)$. Let $\eta_j$ be the probability measure on $\sym(m_j)^G \times \Omega^{[m_j]}$ from which $\bY_j$ is constructed. Let $\pi:\sym(m_j)^G \times \Omega^{[m_j]} \times [m_j] \to \Omega$ be the map $\pi(\sigma,\xi,p) = \xi(p)$ and let $\omega_j = \pi_*(\eta_j \times u_{[m_j]})$ be the pushforward measure. Because $\lim_{j\to\infty} \bY_j = \bY$, $\omega_j$ converges to $\omega$. 

We claim that $h(\bY_j) \le H(\omega_j)$.  Let $\eta_{j,\sigma}$ be the fiber measure of $\eta_j$ over $\sigma \in \sym(m_j)^G$. This is a measure on $\Omega^{[m_j]}$. By abuse of notation we let $\pi:\Omega^{[m_j]} \times [m_j] \to \Omega$ denote the map $\pi(\xi,p) = \xi(p)$ and for $p\in [m_j]$ let $\pi_p:\Omega^{[m_j]} \to \Omega$ denote the map $\pi(\xi) = \xi(p)$.

By concavity of the entropy function,
\begin{eqnarray*}
 \frac{1}{m_j} H(\eta_{j,\sigma}) & \le& \frac{1}{m_j} \sum_{p \in [m_j]} H\left( (\pi_p)_*(\eta_{j,\sigma} ) \right)\\
 &\le& H\left( \frac{1}{m_j} \sum_{p \in [m_j]} (\pi_p)_*(\eta_{j,\sigma} ) \right) = H\left( \pi_*(\eta_{j,\sigma} \times u_{[m_j]})\right).
 \end{eqnarray*}
So by concavity of entropy again,
\begin{eqnarray*}
h(\bY_j) &=& \frac{1}{m_j} H(\eta_j|\kappa_j) =\frac{1}{m_j}  \int H(\eta_{j,\sigma}) ~d\kappa_j(\sigma)\\
&\le &  \int H(\pi_*(\eta_{j,\sigma}\times u_{[m_j]} )) ~d\kappa_j(\sigma) \le H\left( \int \pi_*(\eta_{j,\sigma}\times u_{[m_j]}) ~d\kappa_j(\sigma) \right) = H(\omega_j).
\end{eqnarray*}
So
$$\bh(\sK,\bY)= \lim_{j\to\infty} h(\bY_j) \le \limsup_{j\to\infty} H(\omega_j) =H(\omega). $$
Since we have already shown that $\bh(\sK,\bY) \ge H(\omega)$, we now know that $\bh(\sK,\bY) = H(\omega)$.



\end{proof}

\begin{lem}\label{lem:bhproduct3}
If $\bX$ is any finite-range process over $G$, $\bY=(T,\Omega^G,\omega^G,\phi)$ is a Bernoulli process with base $(\Omega,\omega)$ (where $\Omega$ is finite), $\Omega'$ is a finite set, $\phi':\Omega^G \to \Omega'$, $\beta:\Omega' \to \Omega$ are Borel maps such that $\phi=\beta \circ \phi'$,  $\bY'=(T,\Omega^G,\omega^G,\phi')$ and $\sK$ is a random sofic approximation to $G$ then
\begin{eqnarray*}
h(\sK, \bX \times \bY'| \bY') = h(\sK,\bX), \quad \bh(\sK, \bX \times \bY'| \bY') = \bh(\sK,\bX).
\end{eqnarray*}
\end{lem}

\begin{proof}
We will only prove the statement for upper sofic entropy as the statement for lower sofic entropy is similar. By Lemma \ref{lem:bhproduct}, it suffices to prove $\bh(\sK, \bX \times \bY'| \bY') \ge \bh(\sK,\bX)$. Let $\sK=\{\kappa_i\}_{i=1}^\infty$ where $\kappa_i$ is a measure on $\sym(m_i)^G$. Let $F:\sym(m_i)^G \to \sym(m_i)^G$ be the map $F(\sigma)(g)=\sigma(g)$ if $g \ne e$ and $F(\sigma)(e)$ is the identity permutation. Let $\kappa'_i:=F_*\kappa_i$. Observe that $\sK':=\{\kappa'_i\}_{i=1}^\infty$ is asymptotic to $\sK$. Therefore, by Theorem \ref{thm:asymptotic2r}, we may assume without loss of generality that $\sK=\sK'$.

Define $\bX_j, \eta_j,\kappa_j$ (for $j=1,2,\ldots$) as in the proof of the previous lemma. For $V \subset G$ finite with $e\in V$, $\sigma \in \sym(m_j)^G$ and $\zeta \in \Omega^{[m_j]}$, define $\zeta[\sigma,V] \in (\Omega^V)^{[m_j]}$ by
$$\zeta[\sigma,V](p) := [v \in V \mapsto \zeta(\sigma(v)p)].$$
Let $\teta_j^V$, $\eta_j^V$ be the measures on $\sym(m_j)^G \times A^{[m_j]} \times (\Omega^V)^{[m_j]}$ obtained by pushing $\eta_j\times \omega^{[m_j]}$ forward under the maps
$$(\sigma,\xi,\zeta) \mapsto (\sigma,\xi,\zeta[\sigma,V]),\quad (\sigma,\xi,\zeta) \mapsto (\sigma,\zeta[\sigma,V])$$
respectively. Let $\bY_j^V, \bZ^V_j$ be the approximate process constructed from $\eta_j^V, \teta_j^V$ respectively. Also let $\bY_j = \bY_j^{\{e\}}$. As in the previous lemma we obtain
$$\lim_{j\to\infty} \bZ^V_j = \bX\times \bY^V, \quad \lim_{j\to\infty} \bY^V_j = \bY^V$$
where $\bY^V=(T,\Omega^G, \omega^G, \phi^V)$. 

Suppose that $\beta':\Omega^V \to \Omega'$ is a map such that $\beta\beta':\Omega^V \to \Omega$ is the map at the identity element. Let $id$ denote the identity map on $A$. Then 
$$\lim_{j\to \infty}  id\times \beta' \circ \bZ^V_j = \bX \times ( \beta' \circ \bY^V )$$
and because the map $(\sigma,\xi,\zeta) \mapsto (\sigma, \xi, \beta'\zeta[\sigma,V])$ is injective, $h(id\times \beta' \circ \bZ^V_j ) = h(\bZ_j)$ which implies (by the previous lemma)
$$\lim_{j\to\infty} h(id\times \beta' \circ \bZ^V_j ) = h(\sK,\bX \times \bY) = h(\sK,\bX) + H(\omega).$$
Also the map $(\sigma,\zeta) \mapsto (\sigma,  \beta'\zeta[\sigma,V])$ is injective. So $h(\beta' \circ \bY^V_j) =h(\bY_j)$. By definition, $\bY_j$ is the approximate process constructed from $\kappa_j \times \omega^{[m_j]}$. So $h(\bY_j)=H(\omega)$.

Because $\phi=\beta \circ \phi'$, there exist a sequence $\{V_k\}_{k=1}^\infty$ of finite subsets of $G$ (with $e \in V_k$ for all $k$) and a sequence of maps $\beta'_k:\Omega^{V_k} \to \Omega'$ such that $\beta\beta'_k:\Omega^{V_k} \to \Omega$ is the evaluation map at the identity element and $\beta'_k\phi^{V_k}$ limits on $\phi'$ in the following sense: 
$$\lim_{k\to\infty} \mu\left( \{ x\in X:~ \beta'_k\phi^{V_k}(x) = \phi'(x)\} \right) = 1.$$
Note that
\begin{eqnarray*}
\lim_{k\to\infty}\lim_{j\to\infty} id \times \beta'_k \circ \bZ^{V_k}_j  &=& \lim_{k\to\infty} id \times \beta'_k \circ \bX\times \bY^{V_k} = \bX\times \bY' \\
\lim_{k\to\infty} \lim_{j\to\infty} h( id \times \beta'_k \circ \bZ^{V_k}_j ) &=&h(\sK,\bX) + H(\omega).
\end{eqnarray*}
So a diagonalization argument implies that, without loss of generality we may assume $\{V_k\}_{k=1}^\infty$ and $\{\beta'_k\}_{k=1}^\infty$ are chosen so that there is an increasing sequence $\{i(j)\}_{j=1}^\infty$ of positive integers with
\begin{eqnarray*}
\lim_{j\to\infty}  id \times \beta'_{i(j)} \circ\bZ^{V_{i(j)}}_{i(j)} &=& \bX\times \bY' \\
\lim_{j\to\infty} h( id \times \beta'_{i(j)} \circ \bZ^{V_{i(j)}}_{i(j)} ) &=&h(\sK,\bX) + H(\omega).
\end{eqnarray*}
Let $\pi:A \times \Omega' \to \Omega'$ be the projection map. Note $ \beta'_{i(j)} \circ \bY^{V_{i(j)}}_{i(j)} = \pi \circ (id \times \beta'_{i(j)}) \circ \bZ^{V_{i(j)}}_{i(j)}$. So $h( \pi \circ (id \times \beta'_{i(j)}) \circ \bZ^{V_{i(j)}}_{i(j)}) = H(\omega)$. Therefore,
$$\lim_{j\to\infty} h( id \times \beta'_{i(j)} \circ \bZ^{V_{i(j)}}_{i(j)} )  - h( \pi \circ (id \times \beta'_{i(j)}) \circ \bZ^{V_{i(j)}}_{i(j)})  = h(\sK,\bX).$$
Thus $h(\sK, \bX \times \bY'| \bY') \ge h(\sK,\bX)$ as required.

\end{proof}

Next, we extend Theorem \ref{thm:asymptotic2} to relative entropy:
\begin{thm}\label{thm:asymptotic2r}
Let $\bX$ be a process over a group $G$ with random sofic approximations $\sK$ and $\sL$. We assume $A$ is the range of $\bX$ and $\beta:A \to B$ is a map (both $A$ and $B$ are finite). If $\sK$ and $\sL$ are asymptotic then $\bh(\sK,\bX|\bX \circ \beta)=\bh(\sL,\bX|\bX \circ \beta)$ and $h(\sK,\bX|\bX \circ \beta) = h(\sL,\bX|\bX \circ \beta)$.
\end{thm}
In order to prove this, we will need a few lemmas. We say that elements $\sigma,\sigma' \in \sym(m)^G$ are {\em conjugate} if there exists an element $\tau \in \sym(m)$ such that for all $g\in G$, $\sigma(g) = \tau \sigma'(g)\tau^{-1}$. We say that probability measures $\kappa, \kappa'$ on $\sym(m)^G$ are conjugate if there exists a probability measure $\vartheta$ on $\sym(m)^G\times \sym(m)^G$ with marginals $\kappa$, $\kappa'$ such that $\vartheta$ is supported on the set of all conjugate pairs $(\sigma,\sigma') \in \sym(m)^G\times \sym(m)^G$.

\begin{lem}
Let $\kappa,\kappa'$ be conjugate probability measures on $\sym(m)^G$. Let $\eta$ be a probability measure on $\sym(m)^G\times A^{[m]}$ with projection $\kappa$ and let $\bX$ be the approximate process constructed from $\eta$. Let $\beta:A \to B$ be a map to a finite set $B$. Let $W \subset  G$ be finite.
Then there is an approximate process $\bX'$ constructed from a measure $\eta'$ on $\sym(m)^G\times A^{[m]}$ such that
\begin{enumerate}
\item the projection of $\eta'$ to $\sym(m)^G$ is $\kappa'$;
\item $d_W(\bX,\bX') =0$,
\item $h(\bX'|\beta \circ \bX') \ge h(\bX|\beta \circ \bX)$.
\end{enumerate}
Moreover, if $\bX=(T,X,\mu,\phi)$ is a process over $G$ then for any $\epsilon>0$,
\begin{eqnarray}\label{eqn:strong}
\eta'\Big(\big\{ (\sigma,\psi):~d_W( (\sigma,\psi), \phi) \le \epsilon \big\}\Big) = \eta\Big(\big\{ (\sigma,\psi):~d_W( (\sigma,\psi), \phi) \le \epsilon \big\}\Big).
\end{eqnarray}
\end{lem}

\begin{proof}
{\bf Case 1}. Suppose $\kappa'$ is supported on a singleton $\{\sigma\}$. Then $\kappa$-a.e. $\sigma'$ is conjugate to $\sigma$, so $\kappa$ is supported on a finite set which we denote by $\{\sigma_1,\ldots,\sigma_n\}$. For each $i$ there is an element $\tau_i \in \sym(m)$ such that $\sigma = \tau_i \sigma_i \tau_i^{-1}$. Let $\eta'$ be the measure on $\sym(m)^G \times A^{[m]}$ defined by
$$\eta' \Big( \big\{ (\sigma,\psi) \big\}\Big) := \eta\Big( \big\{ (\sigma_i, \psi \circ \tau_i):~ 1\le i \le n\big\}\Big).$$
Then $\eta'$ projects to $\kappa'$ and 
$$\psi(\sigma(g) p ) = \psi \circ \tau_i \left(\sigma_i(g)\tau_i^{-1}p\right),\quad \forall p \in [m], \psi \in A^{[m]}, 1\le i\le n, g\in G$$
implies $d_W(\bX,\bX')=0$ where $\bX'$ is the approximate process constructed from $\eta'$. It also implies (\ref{eqn:strong}).

For $\gamma \in B^{[m]}$ with $\eta'( \{ (\sigma,\psi') :~ \beta \psi' = \gamma \})>0$, let $\eta'_\gamma$ be the measure on $A^{[m]}$ defined by
\begin{displaymath}
\eta'_\gamma(\{\psi\}) =\left\{ \begin{array}{ll}
 \frac{\eta'( \{ (\sigma,\psi) \}) }{ \eta'( \{ (\sigma,\psi') :~ \beta \psi' = \gamma \}) } & \textrm{ if } \beta\psi = \gamma\\
  0 & \textrm{ otherwise.} \end{array} \right.
  \end{displaymath}
Then
\begin{eqnarray*}
h(\bX'|\beta\circ \bX') &=& \frac{1}{m} \sum_{\gamma} H(\eta'_\gamma)  \eta'( \{ (\sigma,\psi') :~ \beta \psi' = \gamma \}).
\end{eqnarray*}
If  $\eta( \{ (\sigma_i,\psi' \tau_i) :~ \beta \psi' = \gamma \})>0$ then let $\eta_{i,\gamma}$ be the measure on $A^{[m]}$ defined by
\begin{displaymath}
\eta_{i,\gamma} (\{\psi\}) =\left\{ \begin{array}{ll}
 \frac{\eta( \{ (\sigma_i,\psi \tau_i) \}) }{ \eta( \{ (\sigma_i,\psi' \tau_i) :~ \beta \psi' = \gamma \}) } & \textrm{ if } \beta\psi = \gamma\\
  0 & \textrm{ otherwise.} \end{array} \right.
  \end{displaymath}
  Observe that
$$\eta'_\gamma = \frac{\sum_{i=1}^n  \eta( \{ (\sigma_i,\psi'\tau_i) :~ \beta \psi' = \gamma \}) \eta_{i,\gamma}}{ \eta'( \{ (\sigma,\psi') :~ \beta \psi' = \gamma \}) }.$$
By concavity of entropy,
\begin{eqnarray*}
H(\eta'_\gamma ) &\ge &  \frac{\sum_{i=1}^n  \eta( \{ (\sigma_i,\psi'\tau_i) :~ \beta \psi' = \gamma  \}) H(\eta_{i,\gamma })}{ \eta'( \{ (\sigma,\psi') :~ \beta \psi' = \gamma \}) }
\end{eqnarray*}
which implies
\begin{eqnarray*}
h(\bX'|\beta\circ \bX') &\ge & \frac{1}{m} \sum_{i=1}^n \sum_{\gamma} \eta( \{ (\sigma_i,\psi'\tau_i) :~ \beta \psi' = \gamma \}) H(\eta_{i,\gamma})\\
&=& \frac{1}{m}  \sum_{i=1}^n \sum_{\gamma} \eta( \{ (\sigma_i,\psi'\tau_i) :~ \beta \psi' = \gamma\tau_i^{-1} \}) H(\eta_{i,\gamma\tau_i^{-1}})= h(\bX|\beta \circ \bX).
\end{eqnarray*}
This proves Case 1.

{\bf Case 2}. Suppose $\kappa$ is supported on the singleton $\{\sigma\}$. Then $\kappa'$-a.e. $\sigma'$ is conjugate to $\sigma$, so $\kappa'$ is supported on a finite set which we will denote by $\{\sigma_1,\ldots,\sigma_n\}$. For each $i$ there is an element $\tau_i \in \sym(m)$ such that $\sigma_i = \tau_i \sigma \tau_i^{-1}$. Let $\eta'$ be the measure on $\sym(m)^G \times A^{[m]}$ defined by
$$\eta' ( \{ (\sigma_i,\psi) \}) := \kappa'(\{\sigma_i\}) \eta( \{ (\sigma, \psi \circ \tau_i):~ 1\le i \le n\}).$$
Then $\eta'$ projects to $\kappa'$ and 
$$\psi(\sigma_i(g) p ) = \psi \circ \tau_i (\sigma(g)\tau_i^{-1}p),\quad \forall p \in [m], \psi \in A^{[m]}, 1\le i\le n, g\in G$$
implies $d_W(\bX,\bX')=0$ where $\bX'$ is the approximate process constructed from $\eta'$. It also implies (\ref{eqn:strong}). If $\eta'_i$ is the probability measure obtained from $\eta'$ by restricting to $\{\sigma_i\} \times A^{[m]}$ and normalizing to have total mass $1$ and if $\bX'_i$ is the approximate process constructed from $\eta'_i$ then $\bX$ is isomorphic to $\bX'_i$ in the obvious sense. Thus $h(\bX|\beta\circ \bX) = h(\bX'_i|\beta\circ \bX'_i)$. So
\begin{eqnarray*}
h(\bX'|\beta \circ \bX') &=& \sum_{i=1}^n h(\bX'_i|\beta\circ \bX'_i) \kappa'(\{\sigma'_i\})= h(\bX|\beta\circ \bX).
\end{eqnarray*}
This proves Case 2.

{\bf Case 3}. Suppose that there is an element $\sigma \in \sym(m)^G$ such that $\kappa$-a.e. $\sigma'$ is conjugate to $\sigma$. Then $\kappa'$-a.e. $\sigma'$ is also conjugate to $\sigma$. By case 1, there exists a measure $\eta_0$ on $\sym(m)^G \times A^{[m]}$ whose first marginal is the Dirac measure supported on $\{\sigma\}$ such that if $\bX_0$ is the approximate process constructed from $\eta_0$ then $d_W(\bX,\bX_0) = 0$, equation (\ref{eqn:strong}) holds with $\eta'$ replaced by $\eta_0$ and $h(\bX_0|\beta\circ \bX_0) \ge h(\bX|\beta \circ \bX)$. By Case 2, there exists a measure $\eta'$ on $\sym(m)^G \times A^{[m]}$ whose first marginal is $\kappa'$ such that if $\bX'$ is the approximate process constructed from $\eta'$ then $d_W(\bX',\bX_0) = 0$, equation  (\ref{eqn:strong}) holds with $\eta$ replaced by $\eta_0$ and $h(\bX'|\beta\circ \bX') \ge h(\bX_0|\beta \circ \bX_0)$. So $d_W(\bX',\bX)=0$ and $h(\bX'|\beta\circ \bX') \ge h(\bX|\beta \circ \bX)$ and equation (\ref{eqn:strong}) holds.

{\bf Case 4}. Now we handle the general case. This case follows from the previous one by disintegrating $\kappa$ and $\kappa'$ over the set of conjugacy classes. To be precise, let $[\sym(m)^G]$ be the space of conjugacy classes of $\sym(m)^G$. This is the quotient of $\sym(m)^G$ by the conjugacy action of $\sym(m)$. Let $\pi:\sym(m)^G  \to [\sym(m)^G]$ be the quotient map. For each conjugacy class $c \in [\sym(m)^G]$, let $\kappa_c,\kappa'_c$ be the fiber measures of $\kappa,\kappa'$ over $c$ respectively. Likewise, let $\eta_c$ be the fiber measure of $\eta$ over $c$ (so $\kappa_c$ is the projection of $\eta_c$ to $\sym(m)^G$).

 By case 3, if $\bX_c$ is the approximate process constructed from $\eta_c$ then there is a measure $\eta'_c$ on $\sym(m)^G \times A^{[m]}$ such that the projection of $\eta'_c$ to $\sym(m)^G$ is $\kappa'_c$, if $\bX'_c$ is the approximate process constructed from $\eta'_c$ then $d_W(\bX_c,\bX'_c)=0$, equation (\ref{eqn:strong}) holds for $\eta_c$ and $\eta'_c$ in place of $\eta$ and $\eta'$, and $h(\bX'_c|\beta \circ \bX'_c) \ge h(\bX_c|\beta \circ \bX_c)$.
 
 The hypothesis that $\kappa$ and $\kappa'$ are conjugate implies $\pi_*\kappa = \pi_*\kappa'$. Let $\eta'=\int \eta'_c ~d\pi_*\kappa(c)$ and $\bX'$ be the approximate process constructed from $\eta'$. So,
 \begin{eqnarray*}
 h(\bX'|\beta \circ \bX') &=& \int h(\bX'_c|\beta \circ \bX'_c)~d\pi_*\kappa(c)\\
 &\ge&  \int h(\bX_c|\beta \circ \bX_c)~d\pi_*\kappa(c) = h(\bX|\beta \circ \bX).
 \end{eqnarray*}
 Because $d_W(\bX_c,\bX'_c)=0$, it follows that $d_W(\bX,\bX')=0$. Also equation (\ref{eqn:strong}) holds. This finishes the lemma. 

\end{proof}

\begin{lem}
Let $\kappa,\kappa'$ be probability measures on $\sym(m)^G, \sym(m')^G$ respectively. Let $\eta$ be a probability measure on $\sym(m)^G\times A^{[m]}$ with projection $\kappa$ and let $\bX$ be the approximate process constructed from $\eta$. Suppose $\kappa$ and $\kappa'$ are $(W,\epsilon)$-close where $W \subset G$ is a finite set containing the identity and $0\le \epsilon < 1$. Let $\beta:A \to B$ be a map to a finite set $B$.
Then there is an approximate process $\bX'$ constructed from a measure $\eta'$ on $\sym(m')^G\times A^{[m']}$ such that
\begin{enumerate}
\item the projection of $\eta'$ to $\sym(m')^G$ is $\kappa'$;
\item $d_W(\bX,\bX') \le 6|W|\epsilon$,
\item $h(\bX'|\beta \circ \bX') \ge (1-\epsilon)^2h(\bX|\beta\circ\bX) -2 (1-\epsilon)^2 \epsilon\log|A|$.
\end{enumerate}
Moreover, if $\bX=(T,X,\mu,\phi)$ is a process over $G$ then for any $\epsilon>0$,
\begin{eqnarray*}
\eta'\Big(\big\{ (\sigma,\psi):~d_W( (\sigma,\psi), \phi) \le (4|W|+1)\epsilon \big\}\Big)&\ge& (1-\epsilon)\eta\Big(\big\{ (\sigma,\psi):~d_W( (\sigma,\psi), \phi) \le \epsilon \big\}\Big) - \epsilon.
\end{eqnarray*}
\end{lem}

\begin{proof}
Let $\vartheta$ be as in Definition \ref{defn:close}.  So $\vartheta(\cG(W,\epsilon)) \ge 1-\epsilon$ where  $\cG(W,\epsilon)$ is the set of all $(\sigma,\sigma') \in \sym(m)^G \times \sym(m')^G$ that are $(W,\epsilon)$-close to each other. Thus, for every $(\sigma,\sigma') \in \cG(W,\epsilon)$ there exists a bijection $\beta_{\sigma,\sigma'}:Q_{\sigma,\sigma'} \to Q'_{\sigma,\sigma'}$ between subsets $Q_{\sigma,\sigma'} \subset [m]$, $Q'_{\sigma,\sigma'} \subset [m']$ such that 
\begin{enumerate}
\item $\sigma'(w)\beta_{\sigma,\sigma'}(q) = \beta_{\sigma,\sigma'}(\sigma(w)q)$ for all $w \in W$ and $q\in Q_{\sigma,\sigma'}$ with $\sigma(w)q \in Q_{\sigma,\sigma'}$;
\item $\sigma(w)\beta_{\sigma,\sigma'}^{-1}(q') = \beta_{\sigma,\sigma'}^{-1}(\sigma'(w)q')$ for all $w \in W$ and $q'\in Q'_{\sigma,\sigma'}$ with $\sigma'(w)q' \in Q'_{\sigma,\sigma'}$;
\item $|Q_{\sigma,\sigma'}| \ge (1-\epsilon)m$, $|Q'_{\sigma,\sigma'}| \ge (1-\epsilon)m'$.
\end{enumerate}
After removing elements from each $Q_{\sigma,\sigma'}$ and $Q'_{\sigma,\sigma'}$ if necessary, we may assume, without loss of generality, that there exists a number $m'' \le \min(m,m')$ so that $m'' \ge (1-\epsilon)\max(m,m')$ and $m''=|Q_{\sigma,\sigma'}|=|Q'_{\sigma,\sigma'}|$ for every $(\sigma,\sigma') \in \cG(W,\epsilon)$. There exist elements $\tau_{\sigma,\sigma'} \in \sym(m)$ and $\tau'_{\sigma,\sigma'} \in \sym(m')$ such that $\tau_{\sigma,\sigma'}Q_{\sigma,\sigma'} = [m'']$, $\tau'_{\sigma,\sigma'}Q'_{\sigma,\sigma'} = [m'']$ and $\tau'_{\sigma,\sigma'}\beta_{\sigma,\sigma'}\tau_{\sigma,\sigma'}^{-1}$ is the identity map on $[m'']$. Let $\vartheta'$ be the measure obtained from $\vartheta$ by pushing forward under the map defined by: $(\sigma,\sigma') \mapsto (\tau_{\sigma,\sigma'} \sigma \tau_{\sigma,\sigma'}^{-1}, \tau'_{\sigma,\sigma'} \sigma (\tau'_{\sigma,\sigma'})^{-1})$ for $(\sigma,\sigma') \in \cG(W,\epsilon)$ and the map is equal to the identity on the complement of $\cG(W,\epsilon)$. The marginals of $\vartheta'$ are conjugate to $\kappa$ and $\kappa'$ respectively. By the previous lemma, therefore, we may assume that $\vartheta'=\vartheta$.

About notation: if $\mu$ is a measure on a space $X$ and $\pi:X \to Y$ is a Borel map, we let $\mu_y$ denote the fiber measure of $\mu$ over $y\in Y$. It is a measure on $X$ supported on $\pi^{-1}(y)$ and $\mu = \int \mu_y ~d\pi_*\mu(y)$. Note that $\mu_y$ depends on $\pi$ but this dependence is left implicit. If $\mu$ is a probability measure then $\mu_y$ is a probability measure for $\pi_*\mu$-a.e. $y\in Y$. If $\alpha:X \to Y$ and $\beta:Y \to Z$ and $\mu$ is a measure on $X$ then $\mu_z = \int \mu_{\alpha(x)}~d\mu_z(x)$ and $\alpha_*(\mu_z)=(\alpha_*\mu)_z$ for $\beta_*\alpha_*\mu$-a.e. $z\in Z$.


{\bf Case 1}. Suppose $\vartheta(\cG(W,\epsilon))=1$. We will prove that there exists a probability measure $\eta'$ on $\sym(m')^G \times A^{[m']}$ which projects to $\kappa'$ such that if $\bX'$ is the approximate process constructed from $\eta'$ then $d_W(\bX',\bX)\le 4\epsilon |W|$, $h(\bX'|\beta \circ \bX') \ge (1-\epsilon) h(\bX|\beta\circ\bX) - (1-\epsilon) \epsilon\log|A|$ and
 \begin{eqnarray*}
\eta'\Big(\big\{ (\sigma,\psi):~d_W( (\sigma,\psi), \phi) \le (4|W|+1)\epsilon \big\}\Big) \ge \eta\Big(\big\{ (\sigma,\psi):~d_W( (\sigma,\psi), \phi) \le \epsilon \big\}\Big).
\end{eqnarray*}
Let $\pi:\sym(m)^G \to \sym(m'')^G$, $\pi':\sym(m')^G \to \sym(m'')^G$ be Borel maps satisfying
\begin{enumerate}
\item if $\sigma \in \sym(m)^G, w\in W, p \in [m'']$ and $\sigma(w)p \in [m'']$ then $\pi(\sigma)(w)p = \sigma(w)p$;
\item if $\sigma' \in \sym(m')^G, w\in W, p \in [m'']$ and $\sigma'(w)p \in [m'']$ then $\pi'(\sigma')(w)p = \sigma'(w)p$;
\item if $\sigma \in \sym(m)^G$ and $\sigma'(m')^G$ satisfy $\sigma(w)p = \sigma'(w)p$ for every $w\in W$ and $p\in [m'']$ such that $\sigma(w)p \in [m'']$ and $\sigma'(w)p \in [m'']$ then $\pi(\sigma)=\pi'(\sigma')$.
\end{enumerate}
Because we assume $\vartheta(\cG(W,\epsilon))=1$, it follows from these assumptions that $\pi_*\kappa=\pi'_*\kappa'$. Let $\kappa''$ denote $\pi_*\kappa$. Let $R:A^{[m]} \to A^{[m'']}$ be the restriction map.  Let $\eta'' = (\pi \times R)_*\eta$ be the pushforward measure and $\bX''$ be the approximate process constructed from $\eta''$.


Note that if $\zeta$ is any probability measure on $A^{[m]}$ then $H(R_*\zeta) \ge H(\zeta) - (m-m'')\log|A|$. Thus,
\begin{eqnarray*}
h(\bX|\beta\circ\bX) &=& \frac{1}{m} \int H(\eta_{\sigma,\beta\psi}) ~d\eta(\sigma,\psi)\\
&\le & \frac{1}{m} \int H( (id \times R)_*(\eta_{\sigma,\beta\psi})) ~d\eta(\sigma,\psi) + \frac{m-m''}{m}\log|A|
\end{eqnarray*}
where $id$ denotes the identity map on $\sym(m)^G$. By the definition of fiber measure,
$$\int \eta_{\sigma,\beta\psi'} ~d\eta_{\sigma,\beta R \psi}(\sigma,\psi') = \eta_{\sigma,\beta R \psi}.$$
Therefore,
$$\int (id \times R)_*(\eta_{\sigma,\beta\psi'}) ~d\eta_{\sigma,\beta R \psi}(\sigma,\psi') = ((id \times R)_*\eta)_{\sigma,\beta R \psi}.$$
So by concavity of entropy,
$$\int H((id \times R)_*(\eta_{\sigma,\beta\psi'})) ~d\eta_{\sigma,\beta R \psi}(\sigma,\psi') \le H( ((id \times R)_*\eta)_{\sigma,\beta R \psi}).$$
By integrating over all $(\sigma,\psi) \in \sym(m)^G \times A^{[m]}$, we obtain
$$\int H( (id \times R)_*(\eta_{\sigma,\beta\psi})) ~d\eta(\sigma,\psi) \le \int H\Big( ( (id \times R)_*\eta)_{\sigma,\beta R \psi} \Big)~d \eta(\sigma,\psi).$$
Note $\pi\times id $ is injective on the support of $((id \times R)_*\eta)_{\sigma,\beta R \psi} $. So
$$ H\Big( ( (id \times R)_*\eta)_{\sigma,\beta R \psi} \Big) = H\Big( (\pi \times id)_*\big(( (id \times R)_*\eta)_{\sigma,\beta R \psi}\big) \Big).$$
Note that $((id \times R)_*\eta)_{\sigma,\beta R \psi} = (id \times R)_*(\eta_{\sigma,\beta R \psi})$. Therefore
$$(\pi \times id)_*\big(( (id \times R)_*\eta)_{\sigma,\beta R \psi}\big) = (\pi\times R)_* (\eta_{\sigma,\beta R \psi}).$$
Because $\int \eta_{\sigma',\beta R \psi}~d\eta_{\pi\sigma,\beta R\psi}(\sigma',\psi') =\eta_{\pi \sigma, \beta R \psi}$, it follows that
$$\int (\pi\times R)_* \eta_{\sigma',\beta R \psi}~d\eta_{\pi\sigma,\beta R\psi}(\sigma',\psi') = (\pi\times R)_* (\eta_{\pi \sigma, \beta R \psi}) = \eta''_{\pi\sigma,\beta R \psi}.$$
By concavity of entropy,
$$\int H((\pi\times R)_* \eta_{\sigma',\beta R \psi})~d\eta_{\pi\sigma,\beta R\psi}(\sigma',\psi') \le H(\eta''_{\pi\sigma,\beta R \psi}).$$
So,
$$\int H\Big( ((id \times R)_*\eta)_{\sigma,\beta R \psi} \Big)~d\eta(\sigma,\psi)=\int H((\pi\times R)_* (\eta_{\sigma,\beta R \psi})) ~d\eta(\sigma,\psi) \le \int H(\eta''_{\pi\sigma,\beta R \psi})~d\eta(\sigma,\psi).$$
Since $\frac{m-m''}{m} \le \epsilon$,
\begin{eqnarray*}
h(\bX|\beta\circ\bX) &\le&\epsilon \log|A| + \frac{1}{m}\int H\left( \eta''_{\pi \sigma,\beta R \psi}\right) ~d \eta(\sigma,\psi)\\
&=& \epsilon \log|A| + \frac{1}{m}\int H\left( \eta''_{\sigma'',\beta \psi''}\right) ~d \eta''(\sigma'',\psi'')\\
&=&\epsilon\log|A| + \frac{m''}{m} h(\bX''|\beta \circ \bX'').
\end{eqnarray*}


Next let $R':A^{[m']} \to A^{[m'']}$ be the restriction map. We claim that there exists a probability measure $\eta'$ on $\sym(m')^G\times A^{[m']}$ such that the first marginal of $\eta'$ is $\kappa'$ and $(\pi'\times R')_*\eta'=\eta''$. To see this, let $a \in A$ be an arbitrary element. For $\psi \in A^{[m'']}$, define $\psi_a \in A^{[m']}$ by $\psi_a(p)=\psi(p)$ if $p\in [m'']$ and $\psi_a(p)=a$ if $p \in [m']\setminus[m'']$. For $(\sigma, \psi) \in \sym(m'')^G\times A^{[m'']}$, let $\eta'_{\sigma,\psi}$ be the probability measure on $\sym(m')^G\times A^{[m']}$ whose first marginal is $\kappa'_\sigma$ and has support contained in $\{(\sigma',\psi_a):~\pi'(\sigma')=\sigma\}$. That is, $\eta'_{\sigma,\psi}$ is the direct product of $\kappa'_\sigma$ and the Dirac measure concentrated on $\psi_a$. Then $\eta':=\int \eta'_{\sigma,\psi}~d\eta''(\sigma,\psi)$ satisfies the claim because $\kappa'' = (\pi')_*\kappa'$ implies the first marginal of $\eta'$ is $\kappa'$ as required.

Let $\alpha:A^{[m'']} \to A^{[m']}$ be the map $\alpha(\psi)=\psi_a$. Let $\rho: \sym(m'')^G \times A^{[m'']} \to A^{[m'']}$ be the projection map. By construction, for $\kappa''$-a.e. $\sigma'' \in \sym(m'')^G$, $\eta'_{\sigma''} = \kappa'_{\sigma''} \times \alpha_*\rho_*(\eta''_{\sigma''})$. Therefore, for $\kappa'$-a.e. $\sigma$, $\eta'_\sigma $ is the direct product of $\delta_\sigma$, the Dirac measure concentrated on $\sigma$, with $\alpha_*\rho_*(\eta''_{\pi'\sigma})$. So for $\eta'$-a.e. $(\sigma,\psi)$, $\eta'_{\sigma,\beta R'\psi}=\delta_\sigma \times \alpha_*\rho_*(\eta''_{\pi'\sigma, \beta R' \psi})$ and $\alpha(R'\psi)=\psi$ which implies $\eta'_{\sigma,\beta R'\psi} = \eta'_{\sigma, \beta \psi}$. Because $R'\alpha$ is the identity map, 
\begin{eqnarray*}
(\pi'\times R')_*(\eta'_{\sigma,\beta \psi}) &=& (\pi'\times R')_*( \delta_\sigma \times \alpha_*\rho_*(\eta''_{\pi'\sigma, \beta R' \psi}) ) \\
&=& \delta_{\pi' \sigma} \times \rho_*\eta''_{\pi'\sigma, \beta R' \psi} = \eta''_{\pi'\sigma,\beta R' \psi}.
\end{eqnarray*}


 Let $\bX'$ be the approximate process constructed from $\eta'$. Then 
\begin{eqnarray*}
h(\bX'|\beta \circ \bX') &=& \frac{1}{m'} \int H(\eta'_{\sigma,\beta \psi})~d\eta'(\sigma,\psi)\ge  \frac{1}{m'} \int H(\eta''_{\pi'\sigma,\beta R'\psi})~d\eta'(\sigma,\psi)\\
&=& \frac{1}{m'} \int H(\eta''_{\pi'\sigma,\beta R'\psi})~d\eta''(\pi'\sigma,R'\psi)= \frac{m''}{m'} h(\bX''|\beta \circ \bX'')\\
 &\ge& (1-\epsilon) h(\bX|\beta\circ\bX) - (1-\epsilon) \epsilon\log|A|.
\end{eqnarray*}

Let $\chi:\sym(m)^G \times A^{[m]} \times [m] \to A$ be the map $\chi(\sigma,\psi,p)=\psi(p)$. Define $\chi'':\sym(m'')^G \times A^{[m'']} \times [m''] \to A$ similarly. Then 
$$d_W(\bX,\bX'') = \frac{1}{2} \left\| \chi^W_*( \eta\times u) - (\chi'')^W_*(\eta''\times u'') \right\|_1$$
where $u,u''$ denote the uniform probability measures on $[m]$, $[m'']$ respectively. So if $\Theta$ is any probability measure on $\sym(m)^G \times A^{[m]} \times [m] \times \sym(m'')^G \times A^{[m'']} \times [m'']$ with marginals $\eta \times u$ and $\eta''\times u''$ then $d_W(\bX,\bX'') \le \Theta(S)$ where $S = \{(\sigma,\psi,p,\sigma'',\psi'',p'') :~ \chi^W(\sigma,\psi,p) \ne (\chi'')^W(\sigma'',\psi'',p'')\}$.

Let $\Theta''$ be the pushforward of $\eta \times u''$ under the map $(\sigma, \psi,p) \mapsto (\sigma,\psi,p, \pi\sigma,R\psi, p)$. Let $u_0$ be the uniform probability measure on $[m] \setminus [m'']$. Then 
$$\eta \times u = \left(\frac{m-m''}{m}\right)\eta \times u_0 + \left(\frac{m''}{m}\right) \eta \times u''.$$
So if $\Theta_0 = \eta \times u_0 \times \eta''\times u''$ and 
$$\Theta:= \left(\frac{m-m''}{m}\right) \Theta_0 + \left(\frac{m''}{m}\right)  \Theta''$$
then $\Theta$ has marginals $\eta \times u$ and $\eta''\times u''$. Moreover, $\Theta(S) \le \frac{m-m''}{m} + \frac{m''}{m} \Theta''(S)$ and, we claim, $\Theta''(S) \le |W|\epsilon$. This is because if $U$ is the set of all $(\sigma,\psi,p,\pi\sigma,R\psi,p)$ such that $\sigma(g)p \in [m'']$ for all $g\in W$ then $U \cap S = \emptyset$ and $\Theta''(U) \ge 1- |W|\frac{m-m''}{m}$. Thus 
\begin{eqnarray*}
d_W(\bX,\bX'') &\le& \Theta(S) \le \frac{m-m''}{m} + \frac{m''}{m} \Theta''(S)\\
 &\le& \frac{m-m''}{m} + \frac{m''}{m}\left(\frac{m-m''}{m} \right) |W| \le 2\epsilon|W|.
 \end{eqnarray*}
 Similarly, $d_W(\bX',\bX'') \le 2\epsilon |W|$. So $d_W(\bX,\bX') \le 4\epsilon |W|$. 
 
 By applying the same argument to the Dirac measure concentrated on an arbitrary $(\sigma,\psi) \in \sym(m)^G\times A^{[m]}$ we obtain
 $$d_W((\sigma,\psi), (\pi\sigma,R\psi))  \le 2\epsilon |W|.$$
 A similar statement holds with $(\sigma',\psi') \in \sym(m')^G \times A^{[m']}$. So,
 \begin{eqnarray*}
\eta'\Big(\big\{ (\sigma,\psi):~d_W( (\sigma,\psi), \phi) \le (4|W|+1)\epsilon \big\}\Big) \ge \eta\Big(\big\{ (\sigma,\psi):~d_W( (\sigma,\psi), \phi) \le \epsilon \big\}\Big).
\end{eqnarray*}
This finishes case 1.

{\bf Case 2}. This is the general case. Let $\kappa_1, \kappa'_1$ be the marginals of $\vartheta$ restricted to $\cG(W,\epsilon)$ and normalized to have total mass $1$. Let $\eta_1 = \int \eta_\sigma ~d\kappa_1(\sigma)$. So $\eta_1$ is absolutely continuous to $\eta$. Let $\bX_1$ be the approximate process constructed from $\eta_1$. 

Observe that 
\begin{eqnarray*}
h(\bX_1 | \beta \circ \bX_1)  &=& \frac{1}{m} \int H(\eta_{\sigma,\beta \psi}) d\eta_1(\sigma,\psi)\\
 &=& \frac{1}{m} \frac{1}{\vartheta(\cG(W,\epsilon))} \int_{\cG(W,\epsilon)} \int H(\eta_{\sigma,\beta \psi}) d\eta_{\sigma}(\psi) d\vartheta(\sigma,\sigma') \\
 &\ge& -\epsilon \log|A| + \frac{1}{m} \int \int H(\eta_{\sigma,\beta \psi}) d\eta_{\sigma}(\psi) d\vartheta(\sigma,\sigma') \\
 &=& -\epsilon\log|A| + h(\bX|\beta\circ\bX).
  \end{eqnarray*}
  It is easy to check that $d_W(\bX_1,\bX) \le \epsilon$ and
 \begin{eqnarray*}
\eta_1\Big(\big\{ (\sigma,\psi):~d_W( (\sigma,\psi), \phi) \le \epsilon \big\}\Big) \ge  \eta\Big(\big\{ (\sigma,\psi):~d_W( (\sigma,\psi), \phi) \le \epsilon \big\}\Big) - \epsilon.
\end{eqnarray*}

By case 1, there exists a measure $\eta_1'$ on $\sym(m')^G \times A^{[m']}$ such that $\eta'_1$ projects to $\kappa_1'$ and if $\bX'_1$ is the approximate process constructed from $\eta_1'$ then $d_W(\bX'_1,\bX_1)\le 4\epsilon |W|$, $h(\bX'_1|\beta \circ \bX'_1) \ge (1-\epsilon) h(\bX_1|\beta\circ\bX_1) - (1-\epsilon) \epsilon\log|A|$ and
\begin{eqnarray*}
\eta'_1\Big(\big\{ (\sigma,\psi):~d_W( (\sigma,\psi), \phi) \le (4|W|+1)\epsilon \big\}\Big) \ge \eta_1\Big(\big\{ (\sigma,\psi):~d_W( (\sigma,\psi), \phi) \le \epsilon \big\}\Big).
\end{eqnarray*}

Of course, $\kappa_1'$ is absolutely continuous with respect to $\kappa'$. Let $\eta'$ be any probability measure on $\sym(m')^G\times A^{[m']}$ such that $\eta'$ projects to $\kappa'$ and 
$$\eta'_1 = \int \eta'_\sigma ~d\kappa'_1(\sigma).$$
Let $\bX'$ be the approximate process constructed from $\eta'$. Because $\kappa' \ge (1-\epsilon)\kappa'_1$ it follows that $\eta' \ge (1-\epsilon)\eta'_1$. So
\begin{eqnarray*}
h(\bX'|\beta \circ \bX') &=& \frac{1}{m'} \int H(\eta'_{\sigma,\beta\psi})~d\eta'(\sigma,\psi)\ge  (1-\epsilon) \frac{1}{m'} \int H(\eta'_{\sigma,\beta\psi})~d\eta'_1(\sigma,\psi)\\
&=& (1-\epsilon) h(\bX'_1|\beta \circ \bX'_1)\ge (1-\epsilon)^2 h(\bX_1|\beta\circ\bX_1) - (1-\epsilon)^2 \epsilon\log|A|\\
&\ge& (1-\epsilon)^2h(\bX|\beta\circ\bX) -2 (1-\epsilon)^2 \epsilon\log|A|.
\end{eqnarray*}
Of course, we also have $d_W(\bX'_1,\bX') \le \epsilon$ and because $\eta' \ge (1-\epsilon)\eta'_1$,
 \begin{eqnarray*}
\eta'\Big(\big\{ (\sigma,\psi):~d_W( (\sigma,\psi), \phi) \le(4|W|+1)\epsilon \big\}\Big) \ge (1-\epsilon)\eta'_1\Big(\big\{ (\sigma,\psi):~d_W( (\sigma,\psi), \phi) \le (4|W|+1)\epsilon \big\}\Big).
\end{eqnarray*}
Thus 
$$d_W(\bX,\bX')  \le d_W(\bX,\bX_1) +d_W(\bX_1,\bX_1') + d_W(\bX'_1,\bX') \le 6\epsilon|W|$$
and
\begin{eqnarray*}
\eta'\Big(\big\{ (\sigma,\psi):~d_W( (\sigma,\psi), \phi) \le (4|W|+1)\epsilon \big\}\Big) &\ge& (1-\epsilon)\eta'_1\Big(\big\{ (\sigma,\psi):~d_W( (\sigma,\psi), \phi) \le (4|W|+1)\epsilon \big\}\Big)\\
&\ge& (1-\epsilon)\eta_1\Big(\big\{ (\sigma,\psi):~d_W( (\sigma,\psi), \phi) \le \epsilon \big\}\Big)\\
&\ge& (1-\epsilon)\eta\Big(\big\{ (\sigma,\psi):~d_W( (\sigma,\psi), \phi) \le \epsilon \big\}\Big) - \epsilon.
\end{eqnarray*}

\end{proof}

\begin{proof}[Proof of Theorem \ref{thm:asymptotic2r}]
Let $\{\bX_j\}_{j=1}^\infty$ be a sequence of approximate processes adapted to $\sK'$ (where $\sK'$ is a subsequence of $\sK$) such that $\lim_{j\to\infty} \bX_j = \bX$. Let $\eta_j$ be the probability measure on $\sym(m_j)^G\times A^{[m_j]}$ from which $\bX_j$ is constructed.

Let $\epsilon>0$ and $W \subset G$ be finite. Let $\sL=\{\lambda_j\}_{j=1}^\infty$. By the previous lemma, for all sufficiently large $j$ there exists an approximate process $\bY_j$ adapted to $\lambda_j$ such that $h(\bY_j|\beta \circ \bY_j) \ge (1-\epsilon)^2 h(\bX_j|\beta \circ\bX_j) -2\epsilon\log|A|$, $d_{W}(\bY_j, \bX_j) \le 6|W|\epsilon$ and if $\eta'_j$ is the measure on $\sym(m'_j)^G\times A^{[m'_j]}$ from which $\bY_j$ is constructed then
\begin{eqnarray*}
\eta'_j\Big(\big\{ (\sigma,\psi):~d_W( (\sigma,\psi), \phi) \le (4 |W|+1)\epsilon \big\}\Big)&\ge& (1-\epsilon)\eta_j\Big(\big\{ (\sigma,\psi):~d_W( (\sigma,\psi), \phi) \le \epsilon \big\}\Big) - \epsilon.
\end{eqnarray*}
So if $\{W_n\}_{n=1}^\infty$ is an increasing sequence of finite subsets of $G$ then for every $n$ there is a $J(n)>0$ so that for every $j\ge  J(n)$ there is an approximate process $\bY_j$ adapted to $\lambda_j$ such that $h(\bY_j|\beta \circ \bY_j) \ge (1-\frac{1}{n})h(\bX_j|\beta \circ \bX_j) - \frac{1}{n}$, $d_{W_n}(\bY_j, \bX_j) \le \frac{1}{n}$ and, if $\lim_{j\to\infty} \bX_j = \bX$ strongly then
\begin{eqnarray*}
\eta'_j\Big(\big\{ (\sigma,\psi):~d_{W_n}( (\sigma,\psi), \phi) \le \frac{1}{n} \big\}\Big)&\ge& 1-\frac{1}{n}.
\end{eqnarray*}
Without loss of generality we may assume $\{J(n)\}_{n=1}^\infty$ is an increasing sequence. So $\lim_{n\to\infty} \bY_{J(n)} = \bX$, $\{\bY_{J(n)}\}_{n=1}^\infty$ is adapted to a subsequence of $\sL$ and $\limsup_{n\to\infty} h(\bY_n|\beta \circ \bY_n)\ge \bh(\bX|\beta \circ \bX)$. So $\bh(\sL,\bX|\beta \circ \bX)\ge \bh(\sK,\bX|\beta \circ \bX)$. By symmetry we must in fact have $\bh(\sL,\bX|\beta \circ \bX)=\bh(\sK,\bX|\beta \circ \bX)$ as claimed. Moreover, if  if $\lim_{j\to\infty} \bX_j = \bX$ strongly then  $\lim_{n\to\infty} \bY_{J(n)} = \bX$ strongly.  So $h(\sL,\bX|\beta \circ \bX)=h(\sK,\bX|\beta \circ \bX)$. 
 \end{proof}
 
 Next we extend Proposition \ref{prop:Z} to the relative case:
\begin{prop}\label{prop:Zrelative}
Let $\sK=\{\kappa_i\}_{i=1}^\infty$ be a random sofic approximation of $\Z$. Let $\bX$ be a process over $\Z$ with range $A$ and let $\beta:A \to B$ be a map. Then $\bh(\sK,\bX|\beta \circ \bX) = h(\sK,\bX|\beta \circ \bX)= h(\bX|\beta \circ \bX)$. 
\end{prop}

\begin{proof}
The inequalities
$$\bh(\sK,\bX|\beta \circ \bX) \ge  h(\sK,\bX|\beta \circ \bX)\ge h(\bX|\beta \circ \bX)$$
follow immediately from the definitions and Proposition \ref{prop:Z}. So it suffices to show $h(\bX|\beta \circ \bX) \ge \bh(\sK,\bX|\beta \circ \bX)$. By Theorem \ref{thm:asymptotic2r} and Theorem \ref{thm:Zasymptotic} we may assume $\sK$ is the non-random sofic approximation $\Sigma=\{\sigma_i\}_{i=1}^\infty$ where $\sigma_i:\Z\to \sym(m_i)$ is the homomorphism with $\sigma_i(1)=(1,2,\ldots,m_i)$. 

Let $N$ be a large positive integer. If $m'_i$ is the integer nearest to $m_i$ that is divisible by $N$ and $\sigma'_i:\Z\to \sym(m'_i)$ is the homomorphism with $\sigma'_i(1)=(1,2,\ldots,m'_i)$ then $\{\sigma'_i\}_{i=1}^\infty$ is asymptotic to $\{\sigma_i\}_{i=1}^\infty$. So without loss of generality, we may assume $N$ divides $m_i$ for each $i$.

Let $\{\bX_i\}_{i=1}^\infty$ be a sequence of approximate processes constructed from measures $\eta_i$ on $\sym(m_i)^\Z\times A^{[m_i]}$ (as in definition \ref{defn:eta}) adapted to $\sK'$, a subsequence of $\sK$, such that $\lim_{i\to\infty} \bX_i=\bX$ and $\lim_{i\to\infty} h(\bX_i)-h(\beta \circ\bX_i) = \bh(\sK,\bX|\beta \circ \bX)$. Without loss of generality, we may assume $\sK'=\sK$. So $\eta_i=\delta_i\times \lambda_i$  for some measure $\lambda_i$ on $A^{[m_i]}$ where $\delta_i$ is the probability measure concentrated on $\{\sigma_i\}$.

Fix $i$ for now.  For $a, b \in [m_i]$ let $a+b \in [m_i]$ denote their sum modulo $m_i$. Also let $[a,b]$ be the interval from $a$ to $b$: $[a,b] = \{a, a+1, a+2, \ldots, b\}$. For example, $[m_i-1,1] = \{m_i-1,m_i,1\}$. Let $\pi^{[a,b]}: A^{[m_i]} \to A^{[a,b]}$ be the projection map and $\lambda^{[a,b]}_i=\pi^{[a,b]}_*\lambda_i$ the pushforward measure. For $N>0$ a positive integer, let $\lambda'_{i,N}$ be the product measure 
$$\lambda'_{i,N}:=\lambda_i^{[1,N]}\times \lambda_i^{[N+1,2N]} \times \cdots \times \lambda_i^{[m_i-N+1,m_i]}.$$

Let $\beta^{[a,b]}:A^{[a,b]}\to B^{[a,b]}$ denote the map $(\beta^{[a,b]}\xi)(p)=\beta(\xi(p))$. In order to simplify notation, we write $\beta=\beta^{[a,b]}$ when $[a,b]$ is clear from the context. We claim that 
\begin{eqnarray}\label{claim1}
\frac{H(\lambda'_{i,N}) - H(\beta_*\lambda'_{i,N})}{m_i} \ge \frac{H(\lambda_i) - H(\beta_*\lambda_i) }{m_i} = h(\bX_i) - h(\beta \circ \bX_i).
\end{eqnarray}



For $\xi \in B^{[m_i]}$ and $\zeta \in B^{[a,b]}$ we let $\lambda_i(\cdot|\xi)$ and  $\lambda_i(\cdot|\zeta)$ be the measure defined for $E \subset A^{[m_i]}$ by
$$\lambda_i(E|\xi):=\frac{\lambda_i(E \cap (\beta^{[1,m_i]}\pi^{[1,m_i]})^{-1}(\xi))}{\lambda_i( (\beta^{[1,m_i]}\pi^{[1,m_i]})^{-1}(\xi))}, \quad \lambda_i(E|\zeta):=\frac{\lambda_i(E \cap (\beta^{[a,b]}\pi^{[a,b]})^{-1}(\zeta))}{\lambda_i( (\beta^{[a,b]}\pi^{[a,b]})^{-1}(\zeta))}.$$
Other conditional measures such as $\beta_*\lambda_i(\cdot|\zeta)$ are defined similarly.



 
 Let $\lambda''_{i,N}$ be the measure on $A^{[m_i]}$ satisfying $\beta_*\lambda''_{i,N} = \beta_*\lambda_i$ whose fiber over $\xi \in B^{[m_i]}$ is the measure
 $$\lambda''_{i,N}(\cdot |\xi) = \pi^{[1,N]}_*\lambda_i(\cdot|\xi) \times \pi^{[N+1,2N]}_*\lambda_i(\cdot|\xi) \times \cdots \times  \pi^{[m_i-N+1,m_i]}_*\lambda_i(\cdot|\xi).$$
  
Since $\beta_*\lambda''_{i,N} = \beta_*\lambda_i$ and $H(\lambda''_{i,N}(\cdot |\xi)) \ge H(\lambda_{i}(\cdot |\xi))$ for any $\xi$ we have
$$H(\lambda''_{i,N}) - H(\beta_*\lambda''_{i,N}) \ge H(\lambda_i) - H(\beta_*\lambda_i).$$
Note
\begin{eqnarray*}
H(\lambda''_{i,N}) - H(\beta_*\lambda''_{i,N}) &=&\int  \sum_{j=0}^{m_i/N-1}H\left( \pi^{[jN+1,jN+N]}_*\lambda_i\left(\cdot | \xi\right)\right) ~d\beta_*\lambda_i (\xi)\\
&= &\sum_{j=0}^{m_i/N-1}\iint  H\left( \pi^{[jN+1,jN+N]}_*\lambda_i\left(\cdot | \xi\right)\right) ~d\beta_*\lambda_i (\xi|\zeta)d\beta_*\lambda_i^{[jN+1,jN+N]}(\zeta) \\
&\le &\sum_{j=0}^{m_i/N-1}\int  H\left( \pi^{[jN+1,jN+N]}_*\lambda_i\left(\cdot | \zeta\right)\right) ~d\beta_*\lambda_i^{[jN+1,jN+N]}(\zeta) \\
&=&H(\lambda'_{i,N}) - H(\beta_*\lambda'_{i,N}).
\end{eqnarray*}
The inequality above holds by concavity of entropy. This proves the inequality in (\ref{claim1}). The equality in (\ref{claim1}) holds by definition. 

For each $p \in \{0,\ldots, m_i-1\}$ let $\lambda'_{i,p,N}$ be the product measure 
$$\lambda'_{i,p,N}:=\lambda_i^{[p+1,p+N]}\times \lambda_i^{[p+N+1,p+2N]} \times \cdots \times \lambda_i^{[p+m_i-N+1,p+m_i]}.$$
An argument similar to the one proving (\ref{claim1}) shows
\begin{eqnarray}\label{claim2}
H(\lambda'_{i,p,N}) - H(\beta_*\lambda'_{i,p,N}) \ge H(\lambda_i) - H(\beta_*\lambda_i).
\end{eqnarray}

For $a,b \in [m_i]$, define $\tau^{[a,b]}:A^{[1,b-a+1]} \to A^{[a,b]}$ by $\tau^{[a,b]}(\xi)(p) = \xi(p+a-1)$. Let $\tlambda_i^{[a,b]}$ be the measure defined for sets $E \subset A^{[1,b-a+1]}$ by
$$\tlambda_i^{[a,b]}(E):=\lambda_i^{[a,b]}(\tau^{[a,b]}E).$$
For $\zeta \in B^{[1,b-a+1]}$ let  $\tlambda_i^{[a,b]}(\cdot | \zeta)$ be the measure defined for sets $E \subset A^{[1,b-a+1]}$ by
$$\tlambda_i^{[a,b]}(E|\zeta):=\lambda_i^{[a,b]}(\tau^{[a,b]}E|\tau^{[a,b]} \zeta)$$
where we have abused notation by letting $\tau^{[a,b]}$ denote the analogous map from $B^{[1,b-a+1]} \to B^{[a,b]}$.



Because $H(\tlambda_i^{[p+1,p+N+1]}(\cdot|\zeta)) = H(\lambda_i^{[p+1,p+N+1]}(\cdot|\tau^{[p+1,p+N+1]}\zeta)$,
$$H(\lambda'_{i,p,N}) - H(\beta_*\lambda'_{i,p,N}) = \sum_{j=0}^{m_i/N-1} \int H(\tlambda_i^{[jN+p+1,(j+1)N+p]}(\cdot|\zeta)) ~d\beta_*\tlambda_i^{[jN+p+1,(j+1)N+p]}(\zeta).$$

By (\ref{claim2}),
\begin{eqnarray*}
\frac{H(\lambda_i) - H(\beta_*\lambda_i)}{m_i}&\le& \frac{1}{m_i^2} \sum_{p=0}^{m_i-1}H(\lambda'_{i,p,N}) - H(\beta_*\lambda'_{i,p,N}) \\ 
&= &\frac{1}{Nm_i} \sum_{p=0}^{m_i-1} \int H(\tlambda_i^{[p+1,p+N]}(\cdot|\zeta)) ~d\beta_*\tlambda_i^{[p+1,p+N]}(\zeta).
\end{eqnarray*}
By concavity of entropy, if for $\zeta \in B^{[1,N]}$, $\omega_{i,\zeta}$ is the measure on $A^{[1,N]}$ defined by 
$$\omega_{i,\zeta} =  \frac{ \sum_{p=0}^{m_i-1} \beta_*\tlambda_i^{[p+1,p+N]}(\{\zeta\})\cdot \tlambda_i^{[p+1,p+N]}(\cdot|\zeta) }{  \sum_{p=0}^{m_i-1} \beta_*\tlambda_i^{[p+1,p+N]}(\{\zeta\}) }$$
 then
$$\sum_{p=0}^{m_i-1} \int H(\tlambda_i^{[p+1,p+N]}(\cdot|\zeta)) ~d\beta_*\tlambda_i^{[p+1,p+N]}(\zeta) \le  \sum_{p=0}^{m_i-1} \int H(\omega_{i,\zeta}) ~d\beta_*\tlambda_i^{[p+1,p+N]}(\zeta) .$$

Therefore,
$$\frac{H(\lambda_i) - H(\beta_*\lambda_i)}{m_i} \le  \frac{1}{Nm_i} \sum_{p=0}^{m_i-1} \int H(\omega_{i,\zeta}) ~d\beta_*\tlambda_i^{[p+1,p+N]}(\zeta) .$$
So if $\omega_{\beta,i}$ is the measure on $B^{[1,N]}$ defined by
$$\omega_{\beta,i} := \frac{1}{m_i} \sum_{p=0}^{m_i-1} \beta_*\tlambda_i^{[p+1,p+N]}$$
then

$$\frac{H(\lambda_i) - H(\beta_*\lambda_i) }{m_i}\le \frac{1}{N} \int H(\omega_{i,\zeta}) ~d\omega_{\beta,i}(\zeta).$$

Let $\bX=(T,X,\mu,\phi)$. By construction, $\omega_{i,\zeta}$ converges (as $i\to\infty$) to the measure $\mu_\zeta$ defined for $E \subset A^{[1,N]}$ by 
$$\mu_\zeta(E):=\frac{\mu (\{x \in X:~\phi^N(x) \in E,~(\beta\circ\phi)^N(x) = \zeta\}) }{\mu (\{x \in X:~(\beta\circ\phi)^N(x) = \zeta\}) }.$$
Also, $\omega_{\beta,i}$ converges to $(\beta\circ \phi)^N_*\mu$. Therefore
$$\bh(\sK,\bX|\beta \circ \bX) = \limsup_{i\to\infty} \frac{H(\lambda_i) - H(\beta_*\lambda_i)}{m_i} \le \int \frac{H(\mu_{\zeta})}{N} ~d(\beta\circ \phi)^N_*\mu(\zeta) = \frac{H(\phi^N_*\mu) - H((\beta\circ \phi)^N_*\mu)}{N}.$$
The right hand side converges to $h(\bX|\beta \circ \bX)$ as $N$ tends to infinity. This proves the proposition.

\end{proof}

\section{Orbit equivalence and entropy}\label{sec:oe} 
In this section, we prove Theorem \ref{thm:main} by generalizing a theorem of Rudolph and Weiss which is explained next.


\begin{defn}
Let $G, \Gamma$ be countable discrete groups and let $(X,\cB,\mu)$ be a standard probability space. Let $G \cc^T (X,\cB,\mu)$ and $\Gamma \cc^S (X,\cB,\mu)$ be two probability measure preserving actions with the same orbits. We assume that both actions are essentially free. Let $\rho:\Gamma \times X\to G$ be the cocycle
$$\rho(\gamma,x) := g \Leftrightarrow T_{g}(x) = S_{\gamma}(x).$$
If $\cA \subset \cB$ is a sub-$\sigma$-algebra such that $\rho(\gamma,\cdot)$ is $\cA$-measurable for all $\gamma \in \Gamma$ then the {\em orbit change from $T$ to $S$} is said to be $\cA$-measurable. The smallest such $\sigma$-algebra is called the {\em orbit change $\sigma$-algebra}. 
\end{defn}

The next theorem is proven in [RW00].
\begin{thm}\label{thm:RW}
Suppose $T$ is an essentially free ergodic action of a countable discrete amenable group $G$ and 
$\cA$ is a $T$-invariant sub-$\sigma$-algebra. Suppose also 
that $S$ is essentially free action of $\Gamma$ with the same orbits 
as $T$ (this implies $\Gamma$ is amenable and $S$ is ergodic). Suppose the orbit change from $T$ to $S$ is 
$\cA$-measurable. Then for any finite observable $\phi:X\to A$ we conclude 
$$h(T, \phi |\cA) = h(S, \phi |\cA).$$
\end{thm}

The rest of the paper is devoted to proving a related result:
\begin{prop}\label{prop:orbit}
Let $G$ be an amenable group, $\bX=(T,X,\mu,\phi)$ an essentially free $G$-process with finite range $A$ and $S:(X,\mu)\to (X,\mu)$ be an essentially free measure-preserving Borel automorphism with the same orbits as $T$ (i.e., for $\mu$-a.e. $x\in X$, $\{T_gx:~g\in G\} = \{S^nx:~n\in \Z\}$).

Let $\beta:A \to B$ be a map and suppose the orbit change from $T$ to $S$ is measurable with respect to both the $T$-invariant sub-sigma-algebra generated by $\psi:=\beta\circ\phi$ and the $S$-invariant sub-sigma-algebra generated by $\psi$.  Then for any random sofic approximation $\sK$ to $G$,
$$\bh(\sK,\bX|\beta \circ \bX)= h(\sK, \bX| \beta \circ \bX  )  = h(S,\phi | \beta\circ \phi).$$
\end{prop}

Given the proposition above,  we prove:
\begin{thm}
Let $G$ be a countably infinite amenable group with random sofic approximation $\sK$. Let $\bX=(T,X,\mu,\phi)$ be a $G$-process. Then
$$h(\sK,\bX)=\bh(\sK,\bX)=h(\bX).$$
\end{thm}
Of course, this implies Theorem \ref{thm:main}.

\begin{proof}
We will prove the statement for lower-sofic entropy only, the upper-sofic entropy case is similar. Let $\bY=(S,Y,\nu,\psi')$ be a Bernoulli process over $G$ with base $(B',\omega)$ where $B'$ is a finite set and $\omega$ is not supported on a singleton. This process is weakly mixing and $\psi':Y \to B'$ is generating. By [Dy59, Dy63, CFW81], any ergodic essentially free probability measure preserving action of $\Z$ is orbit-equivalent to $(S,Y,\nu)$. So there exists a weakly mixing automorphism $U:(Y,\nu) \to (Y,\nu)$ with the same orbits as $S$ and a Borel map $\psi'':Y \to B''$ to a finite set $B''$ which generates in the sense that the smallest $U$-invariant sigma-algebra on which $\psi''$ is measurable is the Borel sigma algebra of $(Y,\nu)$ (up to measure zero sets).  

 Let $\psi:Y \to B=B'\times B''$ be the map $\psi(y)=(\psi'(y),\psi''(y))$. Note that the Borel sigma-algebra of $(Y,\nu)$ is the smallest sigma-algebra generated by $\psi$ and the $G$-action (modulo measure 0 sets) which is the smallest sigma-algebra generated by $\psi$ and $U$ (modulo measure 0 sets).

Define $V:X \times Y \to X \times Y$ by
$$V(x,y) = (T_gx, S_gy) \Leftrightarrow Uy=S_gy.$$
Note that the orbit change from $T\times S$ to $V$ is measurable with respect to both the $G$-invariant sub-$\sigma$-algebra generated by $\psi$ and the $V$-invariant sub-$\sigma$-algebra generated by $\psi$. Proposition \ref{prop:orbit} and Lemma \ref{lem:bhproduct3} imply 
\begin{eqnarray}\label{eqn:2}
h_{\mu}(\sK, \phi) = h_{\mu \times \nu}(\sK,  \phi\times \psi | \psi)= h_{\mu \times \nu}(V, \phi \times \psi |\psi).
\end{eqnarray}

Let $\mu = \int \lambda ~d\zeta(\lambda)$ be the ergodic decomposition of $\mu$. Because $\bY$ is weakly mixing,
$$\mu\times \nu = \int \lambda \times \nu ~d\zeta(\lambda)$$ is the ergodic decomposition of $\mu\times \nu$.  It is well-known that the classical entropy of a process equals the integral of the entropies of its ergodic components. So Theorem \ref{thm:RW} implies
\begin{eqnarray*}
h_{\mu \times \nu}(V, \phi \times \psi |\psi)&=& \int h_{\lambda \times \nu}(V, \phi \times \psi |\psi)~d\zeta(\lambda)\\
&=&\int h_{\lambda \times \nu}(T \times S, \phi \times \psi|\psi) ~d\zeta(\lambda)\\
&=& h_{\mu \times \nu}(T \times S, \phi \times \psi|\psi)=h_\mu(T,\phi).
\end{eqnarray*}
So (\ref{eqn:2}) implies $h_{\mu}(\sK, \phi)  = h_\mu(T,\phi)$ as required.
 \end{proof}

\subsection{Lifting factors}\label{sec:proof} 

Proposition \ref{prop:orbit} is proven by ``lifting'' the orbit-equivalence to sofic approximations. But first, we ``lift'' factors that do not necessarily come from composing with a map $\beta:A \to B$. 

\begin{defn}
Let $\{\bX_i\}_{i=1}^\infty$ be a sequence of approximate processes over $G$ constructed from a sequence $\{\eta_i\}_{i=1}^\infty$ of probability measures on $\sym(m_i)^G \times A^{[m_i]}$. If $W \subset G$ is finite, $\sigma\in \sym(m_i)^G$ and $\xi \in A^{[m_i]}$, then let 
$$\xi[\sigma,W] \in (A^W)^{[m_i]}, ~~ \xi[\sigma,W](p)(w) := \xi\big(\sigma(w)p\big).$$
Let $\eta_i^W$ be the measure on $\sym(m_i)^G \times (A^W)^{[m_i]}$ obtained by pushing forward $\eta_i$ under the map 
$$(\sigma,\xi) \in \sym(m_i)^G \times A^{[m_i]} \mapsto (\sigma,\xi[\sigma,W]).$$
 Let $\{\bX^W_i\}_{i=1}^\infty$  denote the sequence of approximate processes constructed from $\eta_i^W$.
\end{defn}
The next lemma is immediate.
\begin{lem}\label{lem:W}
If, in the definition above, $\lim_{i\to\infty} \bX_i = \bX=(T,X,\mu,\phi)$ then $\lim_{i\to\infty} \bX_i^W = (T,X,\mu,\phi^W)$.
\end{lem}

Assume that $\bX=(T,X,\mu,\phi)$ is a process. Let $\psi:X \to B$ be a measurable map into a finite or countable set $B$. For each finite $W \subset G$, let $\psi_W:A^W \to B$ be a measurable function satisfying
\begin{eqnarray*}
\psi_W(\xi)=b &\Rightarrow& \mu\Big(\big\{x\in X:~\psi(x)=b \textrm{ and } \phi^W(x)=\xi \big\}\Big)\\
 &&\ge  \mu\Big(\big\{x\in X:~\psi(x)=c \textrm{ and } \phi^W(x)=\xi \big\}\Big)~\forall c \in B.
 \end{eqnarray*}
Warning: do not confuse $\psi_W:A^W \to B$ with $\psi^W:X \to B^W$. 
 
\begin{defn}
Let $\Lambda$ be a function on the set of finite subsets of $G$. We write $\lim_{W \to G} \Lambda(W)=L$ if for every increasing sequence $\{W_j\}_{j=1}^\infty \subset G$ with $\cup_j W_j=G$, $\lim_{j\to\infty} \Lambda(W_j)=L.$
\end{defn}

\begin{lem}\label{lem:factor}
Let $\bX=(T,X,\mu,\phi)$ be a $G$-process such that $\phi:X\to A$ is generating. Let $\{\bX_i\}_{i=1}^\infty$ be a sequence of approximate processes constructed from a sequence $\{\eta_i\}_{i=1}^\infty$ of probability measures on $\sym(m_i)^G\times A^{[m_i]}$. Suppose that $\lim_{i\to\infty} \bX_i = \bX$. Then
\begin{eqnarray*}
\lim_{i\to\infty} \psi_W \circ \bX_i^W  &=& (T,X,\mu, \psi_W\circ\phi^W),\\
\lim_{W \to G} (T,X,\mu, \psi_W\circ\phi^W) &=& (T,X,\mu,\psi).
\end{eqnarray*}
\end{lem}

\begin{proof}
The first limit follows from the previous lemma and Lemma \ref{lem:composition}. The second limit is a consequence of the fact that $\phi$ is generating. 
\end{proof}

\subsection{Lifting orbit-equivalences}\label{sec:lift} 
The concepts of the previous subsection are used to `lift' orbit-equivalences as follows. Let $\bX=(T,X,\mu,\phi)$ be a process over $G$ with range $A$. Suppose $\Gamma$ is a (possibly different) group and $\bY=(S,X,\mu,\phi)$ is a process over $\Gamma$ with the same orbits as $G$ (up to $\mu$-measure zero). Suppose there is a map $\beta:A \to B$ such that the orbit change from $T$ to $S$ is measurable with respect to the $G$-invariant sub-$\sigma$-algebra generated by $\psi:=\beta \circ \phi$ and that both $\bX$ and $\bY$ are essentially free (e.g., for a.e. $x\in X$ $T_g x \ne T_h x$ if $g\ne h$). Define the cocycle $\rho:\Gamma \times X \to G$ by
$$\rho(\gamma,x)= g \Leftrightarrow S_\gamma x = T_gx.$$

\begin{lem}\label{lem:rho}
For each finite set $W \subset G$ there exists a map $\rho_W:\Gamma \times B^W \to G$ such that for every $\gamma \in \Gamma$ and $g\in G$,
$$\lim_{W\to G} \mu\Big( \big\{x \in X:~ \rho(\gamma, x)=g\big\} \Delta \big\{x \in X:~\rho_W(\gamma, \psi^W(x))=g\big\} \Big) =0.$$
\end{lem}

\begin{proof}
For each $\gamma \in \Gamma$ and $g\in G$ let 
$$X(\gamma,g)=\{x\in X:~ \rho(\gamma,x)=g\}.$$
 Choose orderings $G=\{g_1,g_2,\ldots\}$, $\Gamma=\{\gamma_1,\gamma_2,\ldots\}$ of $G$ and $\Gamma$. For each $n$, let $P_n$ be the smallest partition of $X$ containing $X(\gamma_i,g_j)$ for all $1\le i,j \le n$. Because the orbit change is measurable with respect to the smallest $G$-invariant sigma-algebra $\Sigma$ on which $\psi$ is measurable, the partitions $P_n$ are contained in $\Sigma$ (up to measure zero sets). Therefore for each $n$ there are a finite set $W_n \subset G$ and sets $X_n(\gamma_i,g_j)$ for $1\le i,j \le n$ such that
\begin{enumerate}
\item $X_n(\gamma_i,g_j)$ is contained in the smallest sigma-algebra on which $\psi^{W_n}$ is measurable;
\item $X_n(\gamma_i,g_j) \cap X_n(\gamma_i,g_k) = \emptyset$ if $j \ne k$ (and $1\le i,j,k\le n$);
\item $\mu( X(\gamma_i,g_j) \Delta X_n(\gamma_i,g_j) ) < 2^{-n}$ for all $1\le i,j\le n$.
\end{enumerate}
Let $\rho_{W_n}:\Gamma \times B^{W_n} \to G$ be a map such that $\rho_{W_n}(\gamma_i, \psi^{W_n}(x))  = g_j$ if $1\le i, j\le n$  and $x\in X_n(\gamma_i,g_j)$. 

For any finite $W \subset G$, if there is an $i \ge 1$ such that $W_i \subset W$ then choose a maximal such $i$ and let $\rho_W := \rho_{W_i}$. Otherwise, define $\rho_W$ arbitrarily. The lemma follows immediately.

\end{proof}


 \begin{lem}[Asymptotic cocycle identity]\label{lem:cocycle}
 For any $\gamma_1,\gamma_2\in \Gamma$,
 $$\lim_{W\to G}  \psi^W_*\mu\Big(\big\{\xi\in B^W:~\rho_W(\gamma_1\gamma_2,\xi)=\rho_V(\gamma_1,\xi')\rho_W(\gamma_2,\xi)\big\}\Big)=1$$
 where $g_2=\rho_W(\gamma_2,\xi)$, $V=W \cap Wg_2^{-1}$, $\xi' = \psi^V(T_{g_2}x)$ and $x\in X$ is any element with $\psi^W(x)=\xi$.
  \end{lem}
 \begin{proof}
The claim of the lemma is equivalent to
  $$\lim_{W\to G}  \mu\Big(\big\{x\in X:~\rho_W(\gamma_1\gamma_2,\psi^W(x))=\rho_V(\gamma_1,\psi^V(T_{g_2}x))\rho_W(\gamma_2,\psi^W(x))\big\}\Big)=1$$ 
where $V,g_2$ are as defined above. For $\gamma \in \Gamma, g\in G$ define $X(\gamma,g):=\{x\in X:~\rho(\gamma,x)=g\}$ as in the previous lemma. Also for $W \subset G$ finite define
$$X_W(\gamma,g):=\{x\in X:~\rho_W(\gamma,\psi^W(x))=g\}.$$
The cocycle identity $\rho(\gamma_1\gamma_2,x)=\rho(\gamma_1,S_{\gamma_2} x)\rho(\gamma_2,x)$ is equivalent to the statement:
$$X(\gamma_2, g_2) \cap T_{g_2}^{-1}X(\gamma_1,g_1) \subset X(\gamma_1\gamma_2,g_1g_2)$$
(mod 0) for every $g_1,g_2 \in G$.  Let 
\begin{eqnarray*}
Y(\gamma_1,\gamma_2,g_1,g_2)&:=& X(\gamma_1\gamma_2,g_1g_2) \cap X(\gamma_2,g_2)  \cap T_{g_2}^{-1} X(\gamma_1,g_1 )\\
Y_W(\gamma_1,\gamma_2,g_1,g_2)&:=& X_W(\gamma_1\gamma_2,g_1g_2) \cap X_W(\gamma_2,g_2)  \cap T_{g_2}^{-1} X_V(\gamma_1,g_1 ).
\end{eqnarray*}
Then $\{Y(\gamma_1,\gamma_2,g_1,g_2):~g_1,g_2 \in G\}$ partitions $X$ by the cocycle equation. Observe that
$$Y_W(\gamma_1,\gamma_2,g_1,g_2) \subset \big\{x\in X:~\rho_W(\gamma_1\gamma_2,\psi^W(x))=\rho_V(\gamma_1,\psi^V(T_{g_2}x))\rho_W(\gamma_2,\psi^W(x))\big\}.$$
So it suffices to prove that
$$\lim_{W \to G} \mu\left( \bigcup_{g_1,g_2 \in G} Y_W(\gamma_1,\gamma_2,g_1,g_2) \right) =1.$$
If $V= W \cap Wg_2^{-1}$ then as $W \to G$, $V \to G$. So the previous lemma implies
\begin{eqnarray*}
0&=&\lim_{W \to G}  \mu\big( X_W(\gamma_1\gamma_2,g_1g_2)  \Delta  X(\gamma_1\gamma_2,g_1g_2)  \big)\\
0&=&\lim_{W \to G}  \mu\big( X_W(\gamma_2,g_2)  \Delta  X(\gamma_2,g_2)  \big)\\
0&=&\lim_{W \to G}  \mu\big(T_{g_2}^{-1} X_V(\gamma_1,g_1)  \Delta T_{g_2}^{-1} X(\gamma_1,g_1 ) \big)
\end{eqnarray*}
which implies
$$0 = \lim_{W\to G} \mu(Y(\gamma_1,\gamma_2,g_1,g_2) \Delta Y_W(\gamma_1,\gamma_2,g_1,g_2))$$
for any fixed $g_1,g_2 \in G$. Since $\{Y(\gamma_1,\gamma_2,g_1,g_2):~g_1,g_2 \in G\}$ partitions $X$, for every $\epsilon>0$ there is a finite set $S \subset G$ such that
$$\mu\left( \bigcup_{g_1,g_2 \in S} Y(\gamma_1,\gamma_2,g_1,g_2) \right) \ge 1-\epsilon.$$
Because
$$0=\lim_{W \to G} \mu\left( \bigcup_{g_1,g_2 \in S} Y(\gamma_1,\gamma_2,g_1,g_2) \Delta   Y_W(\gamma_1,\gamma_2,g_1,g_2)  \right)$$
it follows that
$$\lim_{W \to G} \mu\left( \bigcup_{g_1,g_2 \in G} Y_W(\gamma_1,\gamma_2,g_1,g_2) \right) \ge  1-\epsilon.$$
Since $\epsilon>0$ is arbitrary, the lemma follows.

\end{proof}
 
 In general, if $Z$ is a topological space and $\{z_{W,i}\}_{W \subset G, i\in \N}$ is a collection of elements of $Z$ then we write $\lim_{(W,i) \to (G,\infty)} z_{W,i} =z$ if for every limit point $z_W$ of $\{z_{W,i}\}_{i=1}^\infty$, $\lim_{W \to G} z_W = z$.
  
Let $\{\bX_{\beta,i}\}_{i=1}^\infty$ be a sequence of approximate processes constructed from measures $\{\eta_i\}_{i=1}^\infty$ (adapted to a random sofic approximation $\sK$ to $G$) so that $\lim_{i\to\infty} \bX_{\beta,i} = \beta \circ \bX$. 
For each $(\sigma,\xi) \in \sym(m_i)^G\times B^{[m_i]}$ and $\gamma \in \Gamma$, define $S'_{W,\sigma,\xi}(\gamma):[m_i] \to [m_i]$ by
$$S'_{W,\sigma,\xi}(\gamma)(p) := \sigma(g)p \Leftrightarrow \rho_W(\gamma,\xi[\sigma,W](p)) = g.$$
\begin{lem}\label{lem:S0}
 For any $\gamma_1,\gamma_2 \in \Gamma$,
 $$\lim_{(W,i) \to (G,\infty)} \eta_i\times u_{m_i}\left(\{ (\sigma,\xi,p) \in \sym(m_i)^G\times B^{[m_i]}\times [m_i]:~S'_{W,\sigma,\xi}(\gamma_1)S'_{W,\sigma,\xi}(\gamma_2)p=S'_{W,\sigma,\xi}(\gamma_1\gamma_2)p\}\right)=1.$$
 Also for any $g_2' \in G$,
  \begin{eqnarray*}
  &&\lim_{(W,i) \to (G,\infty)} \eta_i\times u_{m_i}\left(\{ (\sigma,\xi,p) \in \sym(m_i)^G\times B^{[m_i]}\times [m_i]:~\rho_W(\gamma_2,\xi[\sigma,W](p)) = g_2'\}\right)\\
  &=&\mu(\{ x\in X:~ \rho(\gamma_2,x)=g'_2\}).
  \end{eqnarray*}
  \end{lem}
  
\begin{proof}

For $g_1,g_2 \in G$, let
$$Y(g_1,g_2):=\{x\in X:~ \rho(\gamma_2,x)=g_2, \rho(\gamma_1,T_{g_2}x)=g_1 \}.$$
Note that a.e. $x\in Y(g_1,g_2)$ satisfies $\rho(\gamma_1\gamma_2,x)=g_1g_2$ by the cocycle equation. For $i \ge 0$ and $W \subset G$ finite let $Y^W_i(g_1,g_2)$ be the set of all $ (\sigma,\xi,p) \in \sym(m_i)^G\times B^{[m_i]}\times [m_i]$ such that
\begin{eqnarray*}
g_1g_2 &=& \rho_W(\gamma_1\gamma_2,\xi[\sigma,W](p)),\\
g_2 &=&\rho_W(\gamma_2,\xi[\sigma,W](p)),\\
g_1&=& \rho_W(\gamma_1,\xi[\sigma,W](\sigma(g_2)p)),\\
\sigma(g_1g_2)p&=&\sigma(g_1)\sigma(g_2)p.
\end{eqnarray*}
Note that if $(\sigma,\xi,p)  \in Y^W_i(g_1,g_2)$ then $S'_{W,\sigma,\xi}(\gamma_1)S'_{W,\sigma,\xi}(\gamma_2)p=S'_{W,\sigma,\xi}(\gamma_1\gamma_2)p$. Therefore, it suffices to show that
\begin{eqnarray}\label{eqn:Y2}
 \lim_{(W,i) \to (G,\infty)} \eta_i\times u_{m_i}\left( \bigcup \left\{Y_i^W(g_1,g_2):~ g_1,g_2 \in G \right\}\right) =1.
 \end{eqnarray}
We claim that for every $g_1,g_2\in G$,
\begin{eqnarray}\label{eqn:Y}
\lim_{(W,i) \to (G,\infty)} \eta_i\times u_{m_i}(Y^W_i(g_1,g_2)) = \mu( Y(g_1,g_2)), \quad \forall g_1,g_2 \in G.
\end{eqnarray}
To see this, for any finite $W \subset G$ and $g_1,g_2 \in G$, let  $Z(W;g_1,g_2)$ be the set of all $x\in X$ such that
\begin{eqnarray*}
g_1g_2 &=& \rho_W(\gamma_1\gamma_2,  \psi^W(x) )\\
g_2 &=&\rho_W(\gamma_2,\psi^W(x) ) \\
g_1&=& \rho_W(\gamma_1, \psi^W(T_{g_2}x) ).
\end{eqnarray*}
Then
$$\lim_{i\to \infty } \eta_i\times u_{m_i}(Y^W_i(g_1,g_2)) = \mu(Z(W;g_1,g_2)).$$
This follows from Lemmas \ref{lem:W} and \ref{lem:composition}. Lemma \ref{lem:rho} implies 
\begin{eqnarray*}
&&\lim_{W\to G} \mu\Big( \big\{x\in X:~g_1g_2 = \rho_W(\gamma_1\gamma_2,  \psi^W(x) )\big\}  \Delta \big\{ x\in X:~ g_1g_2 =  \rho(\gamma_1\gamma_2,  x )\big\}\Big)=0 \\
&&\lim_{W\to G} \mu\Big( \big\{x\in X:~      g_2 = \rho_W(\gamma_2,  \psi^W(x) )\big\}  \Delta \big\{ x\in X:~ g_2 =  \rho(\gamma_2,  x )\big\}\Big)=0 \\
&&\lim_{W\to G} \mu\Big( \big\{x\in X:~      g_1 = \rho_W(\gamma_1,  \psi^W(T_{g_2}x) )\big\}  \Delta \big\{ x\in X:~ g_1 =  \rho(\gamma_1,  T_{g_2}x  )\big\}\Big)=0.
\end{eqnarray*}
Therefore
$$\lim_{W \to G}  \mu(Z(W;g_1,g_2)) = \mu(Y(g_1,g_2))$$
which implies the claim. Because 
$$\mu\left( \bigcup \left\{Y(g_1,g_2):~ g_1,g_2 \in G\right\} \right) =1,$$
for any $\epsilon>0$ there exists a finite set $G' \subset G$ such that 
$$\mu\left( \bigcup \left\{Y(g_1,g_2):~ g_1,g_2 \in G'\right\} \right) \ge 1-\epsilon.$$
By the claim,
$$ \lim_{(W,i) \to (G,\infty)} \eta_i\times u_{m_i}\left( \bigcup \left\{Y_i^W(g_1,g_2):~ g_1,g_2 \in G' \right\}\right) \ge 1-\epsilon.$$
Since this is true for every $\epsilon$,
$$ \lim_{(W,i) \to (G,\infty)} \eta_i\times u_{m_i}\left( \bigcup \left\{Y_i^W(g_1,g_2):~ g_1,g_2 \in G \right\}\right) =1$$
which implies the first statement.

To see the second statement, observe
\begin{eqnarray*}
&&\{ (\sigma,\xi,p) \in \sym(m_i)^G\times B^{[m_i]}\times [m_i]:~\rho_W(\gamma_2,\xi[\sigma,W](p)) = g_2'\} \cap \bigcup \left\{Y_i^W(g_1,g_2):~ g_1, g_2\in G \right\}\\
&=&  \bigcup \left\{Y_i^W(g_1,g_2'):~ g_1 \in G \right\}.
\end{eqnarray*}
Because $Y_i^W(g_1,g'_2)$ is disjoint from $Y_i^W(g'_1,g'_2)$ if $g_1 \ne g'_1$ equations (\ref{eqn:Y2}) and (\ref{eqn:Y}) imply
 \begin{eqnarray*}
 &&\liminf_{(W,i) \to (G,\infty)} \eta_i\times u_{m_i}\left(\{ (\sigma,\xi,p) \in \sym(m_i)^G\times B^{[m_i]}\times [m_i]:~\rho_W(\gamma_2,\xi[\sigma,W](p)) = g_2'\}\right)\\
 &=&\mu \left( \bigcup \left\{Y(g_1,g_2'):~ g_1 \in G \right\}\right)= \mu(\{ x\in X:~ \rho(\gamma_2,x)=g'_2\})
 \end{eqnarray*}
 as required.

   \end{proof}


\begin{lem}\label{lem:S}
There are maps $S_{W,\sigma,\xi}(\gamma) \in \sym(m_i)$ (for every $W,\sigma,\xi,\gamma$) such that for every $\gamma \in \Gamma$,
$$\lim_{W\to G}\liminf_{i\to\infty} \eta_i\times u_{m_i}\big(\{(\sigma,\xi,p)\in \sym(m_i)^G\times B^{[m_i]}\times [m_i]:~S'_{W,\sigma,\xi}(\gamma)(p) = S_{W,\sigma,\xi}(\gamma)(p)\}\big)=1.$$
\end{lem}
\begin{proof}
It follows from  the previous lemma that for any $\gamma \in \Gamma$,
$$\lim_{W\to G}\liminf_{i\to\infty} \eta_i\times u_{m_i}\big(\{(\sigma,\xi,p)\in \sym(m_i)^G\times B^{[m_i]}\times [m_i]:~S'_{W,\sigma,\xi}(\gamma)S'_{W,\sigma,\xi}(\gamma^{-1}) (p) = p\}\big)=1.$$
Therefore, $S'_{W,\sigma,\xi}$ is asymptotically surjective, in the sense that
$$\lim_{W\to G}\liminf_{i\to\infty} \eta_i\times u_{m_i}\big(\{(\sigma,\xi,p)\in \sym(m_i)^G\times B^{[m_i]}\times [m_i]:~p \in \textrm{image } S'_{W,\sigma,\xi}(\gamma)\}\big)=1.$$
This implies the lemma.
\end{proof}


Let $\theta_{W,i}$ be the measure on $\sym(m_i)^\Gamma \times B^{[m_i]}$ obtained by pushing $\eta_i$ forward under the map
$$(\sigma,\xi) \in \sym(m_i)^G\times B^{[m_i]} \mapsto (S_{W,\sigma, \xi}, \xi).$$
Let $\bY_{\beta,W,i}$ be the approximate process constructed from $\theta_{W,i}$ and $\kappa^\rho_{W,i}$ be the projection of $\theta_{W,i}$ to $\sym(m_i)^\Gamma$. 

\begin{lem}\label{lem:lemma0.5}
$\{\kappa^\rho_{W,i}\}_{i\in \N, W\subset G}$ is a random sofic approximation to $\Gamma$ in the following sense. For  every $\gamma_1,\gamma_2 \in \Gamma$
$$\lim_{(W,i) \to (G,\infty)}  \kappa^\rho_{W,i}\times u_{m_i}\left(\{(\sigma,p) \in \sym(m_i)^\Gamma \times [m_i]:~\sigma(\gamma_1)\sigma(\gamma_2)p = \sigma(\gamma_1\gamma_2)p\}\right) = 1$$
and for every $\gamma_1\ne \gamma_2 \in \Gamma$,
$$\lim_{(W,i) \to (G,\infty)}  \kappa^\rho_{W,i}\times u_{m_i}\left(\{(\sigma,p) \in \sym(m_i)^\Gamma \times [m_i]:~\sigma(\gamma_1)p \ne \sigma(\gamma_2)p\}\right) = 1.$$
\end{lem}

\begin{proof}
For the first assertion, note that by definition of $\kappa^\rho_{W,i}$ and the previous lemma
\begin{eqnarray*}
&&\lim_{(W,i) \to (G,\infty)}  \kappa^\rho_{W,i}\times u_{m_i}\left(\{(\sigma,p) \in \sym(m_i)^\Gamma \times [m_i]:~\sigma(\gamma_1)\sigma(\gamma_2)p = \sigma(\gamma_1\gamma_2)p\}\right)\\
&=& \lim_{(W,i) \to (G,\infty)} \eta_i\times u_{m_i}\left(\{ (\sigma,\xi,p) \in \sym(m_i)^G\times B^{[m_i]}\times [m_i]:~S_{W,\sigma,\xi}(\gamma_1)S_{W,\sigma,\xi}(\gamma_2)p=S_{W,\sigma,\xi}(\gamma_1\gamma_2)p\}\right)\\
&=& \lim_{(W,i) \to (G,\infty)} \eta_i\times u_{m_i}\left(\{ (\sigma,\xi,p) \in \sym(m_i)^G\times B^{[m_i]}\times [m_i]:~S'_{W,\sigma,\xi}(\gamma_1)S'_{W,\sigma,\xi}(\gamma_2)p=S'_{W,\sigma,\xi}(\gamma_1\gamma_2)p\}\right)\\
&=&1.
\end{eqnarray*}
The last equality follows from Lemma \ref{lem:S0}.

Because the first statement is true, to prove the second statement it suffices to show that for any $\gamma \ne e$,
$$\lim_{(W,i) \to (G,\infty)}  \kappa^\rho_{W,i}\times u_{m_i}\left(\{(\sigma,p) \in \sym(m_i)^\Gamma \times [m_i]:~p \ne \sigma(\gamma)p\}\right) = 1.$$
Equivalently,
$$\lim_{(W,i) \to (G,\infty)}  \eta_i\times u_{m_i}\left(\{ (\sigma,\xi,p) \in \sym(m_i)^G\times B^{[m_i]}\times [m_i]:~p\ne S_{W,\sigma,\xi}(\gamma)p\}\right)=1.$$
By Lemma \ref{lem:S} and the definition of $S'_{W,\sigma,\xi}$, it suffices to prove
$$\lim_{(W,i) \to (G,\infty)}  \eta_i\times u_{m_i}\left(\{ (\sigma,\xi,p) \in \sym(m_i)^G\times B^{[m_i]}\times [m_i]:~p\ne \sigma(\rho_W(\gamma,\xi[\sigma,W](p)))p\}\right)=1.$$
If $g\in G \setminus\{e\}$ then because $\sK$ is a random sofic approximation to $G$,
$$\lim_{i\to\infty} \eta_i\times u_{m_i}\left(\{ (\sigma,\xi,p) \in \sym(m_i)^G\times B^{[m_i]}\times [m_i]:~p\ne \sigma(g)p\}\right) =1.$$
Also, if $\gamma$ is not the identity element then by Lemma \ref{lem:S0} and because $\Gamma \cc (X,\mu)$ is essentially free, for every $\epsilon>0$ there is a finite set $V \subset G \setminus \{e\}$ such that
$$\liminf_{(W,i)\to(G,\infty)} \eta_i\times u_{m_i}\left(\{ (\sigma,\xi,p) \in \sym(m_i)^G\times B^{[m_i]}\times [m_i]:~\rho_W(\gamma,\xi[\sigma,W](p)) \in V \}\right)\ge 1-\epsilon.$$
Because $\epsilon>0$ is arbitrary this inequality and the previous equation imply the lemma.


\end{proof}


\begin{lem}\label{lem:lemma1}
The following hold.
\begin{enumerate}
\item $\lim_{(W,i)\to (G,\infty)} \bY_{\beta,W,i} = \beta \circ \bY$.
\item If $\lim_{i\to\infty} \bX_{\beta,i} = \beta \circ \bX$ strongly then $\lim_{(W,i) \to (G,\infty)}  \bY_{\beta,W,i} = \beta \circ \bY$ strongly in the sense that for all finite $V \subset \Gamma$ and $\epsilon>0$,
$$\lim_{(W,i)\to (G,\infty)} \theta_{W,i}\Big(\big\{(\sigma,\xi)\in \sym(m_i)^\Gamma\times B^{[m_i]}:~d_V((\sigma,\xi),\beta \circ \phi)<\epsilon \big\}\Big) =1.$$

\end{enumerate}
\end{lem}

\begin{proof}
The first statement is equivalent to stating that the $V$-local statistics of $\bY_{\beta,W,i}$ converges to the $V$-local statistics of $\beta \circ \bY$ for every finite $V \subset \Gamma$. 

For $W \subset G$ finite define
$$L_{V,W,i}:\sym(m_i)^G \times B^{[m_i]} \times [m_i] \to B^V$$
by
$$L_{V,W,i}(\sigma,\xi,p)(\gamma):= \xi\left(S_{W,\sigma,\xi}(\gamma)p\right)\quad \forall \gamma \in V.$$
By definition, the $V$-local statistics of $\bY_{\beta,W,i}$ is $(L_{V,W,i})_*\eta_i\times u_{m_i}$. 

 Let $\omega_{W,i}$ be the $W$-local statistics of $\bX_{\beta,i}$. To be precise, if 
$$\chi_{W,i}:\sym(m_i)^G \times B^{[m_i]} \times [m_i] \to B^W$$
is defined by
$$\chi_{W,i}(\sigma,\xi,p)(g):= \xi\left(\sigma(g)p\right)\quad \forall g \in W$$
then $\omega_{W,i} = (\chi_{W,i})_*\eta_i\times u_{m_i}$. 

Fix $b_0 \in B$. Define $\Phi_{W,V}:B^W \to B^V$  by 
$$\Phi_{W,V}(\xi)(\gamma) := \xi( \rho_W(\gamma,\xi) )$$
if $\gamma \in V$ and $\rho_W(\gamma,\xi) \in W$. Set $\Phi_{W,V}(\xi)(\gamma) :=b_0$ otherwise. We claim that $(\Phi_{W,V})_*\omega_{W,i}$ is asymptotic to $(L_{V,W,i})_*\eta_i\times u_{m_i}$ in the sense that
\begin{eqnarray}\label{claim10}
\lim_{(W,i) \to (G,\infty)} \| (\Phi_{W,V})_*\omega_{W,i} - (L_{V,W,i})_*\eta_i\times u_{m_i} \|_1 = 0.
\end{eqnarray}
According to Lemma \ref{lem:S} and the definition of $S'_{W,\sigma,\xi}$, for all $\gamma \in V$,
$$\lim_{(W,i) \to (G,\infty)} \eta_i\times u_{m_i}\left(\{(\sigma,\xi,p):~S_{W,\sigma,\xi}(\gamma)p = \sigma(\rho_W(\gamma, \xi[\sigma,W](p)))p \}\right)=1.$$
Because $\lim_{i\to\infty} \bX_{\beta,i} =\beta \circ \bX$, $\lim_{i\to\infty} \omega_{W,i} = \psi^W_*\mu$.  So by Lemmas \ref{lem:W} and \ref{lem:rho},  for all $\gamma \in V$,
$$\lim_{(W,i) \to (G,\infty)} \eta_i\times u_{m_i}\left(\{(\sigma,\xi,p):~\rho_W(\gamma, \xi[\sigma,W](p)) \in W\}\right) =1.$$
Therefore,
$$\lim_{(W,i) \to (G,\infty)} \eta_i\times u_{m_i}\left(\{(\sigma,\xi,p):~L_{V,W,i}(\sigma,\xi,p) = \Phi_{W,V}\chi_{W,i}(\sigma,\xi,p)\}\right) =1.$$
This implies (\ref{claim10}).

By Lemmas  \ref{lem:W} and \ref{lem:composition},
$$\lim_{i \to \infty} \Phi_{W,V} \circ \bX_{\beta,i}^W = (T,X,\mu,\Phi_{W,V}\psi^W).$$
Because $\Phi_{W,V}\psi^W_*\mu$ converges to $\psi^V_*\mu$ as $W \to G$,
$$\lim_{(W,i) \to (G,\infty)} \Phi_{W,V} \circ \bX_{\beta,i}^W = (T,X,\mu,\psi^V).$$
By (\ref{claim10}) this means that the $V$-local statistics of $\bY_{\beta,W,i}$ converge to $(\psi^V)_*\mu$ which is the $V$-local statistics of $\beta\circ \bY$. This proves the first statement of the lemma. The second statement is similar. 
\end{proof}




In the previous lemma we used a sequence $\{\bX_{\beta,i}\}_{i=1}^\infty$ of approximate processes over $G$ converging to $\beta \circ \bX$ to construct a sequence $\{\bY_{\beta, W,i}\}_{i\in \N, W\subset G}$ that converges to $\beta \circ \bY$ (there is a slight abuse of notation here since $W$ varies over all {\em finite} subsets of $G$ instead of all subsets; we will continue this abuse below). In the next lemma, a sequence $\{\bX_i\}_{i=1}^\infty$ of approximate processes over $G$ such that $\beta \circ \bX_i=\bX_{\beta,i}$ is used to construct a new sequence $\{\bY_{W,i}\}_{i\in \N, W\subset G}$ of approximate processes over $\Gamma$ satisfying various properties.

\begin{lem}\label{lem:forward}
Given a sequence $\{\bX_i\}_{i=1}^\infty$ of approximate processes over $G$ such that $\beta \circ \bX_i = \bX_{\beta,i}$ there exists a collection $\{\bY_{W,i}\}_{i\in \N, W\subset G}$ of approximate processes over $\Gamma$ such that 
\begin{enumerate}
\item $\beta \circ \bY_{W,i}  $ is equivalent to $ \bY_{\beta,W,i}$. 
\item  if $\lim_{i\to\infty} \bX_i = \bX$ then $\lim_{(W,i) \to (G,\infty)} \bY_{W,i} = \bY.$
\item  If $\lim_{i\to\infty} \bX_i = \bX$ strongly then $\lim_{(W,i) \to (G,\infty)} \bY_{W,i} = \bY$ strongly.
\item $h(\bY_{W,i}) - h(\beta \circ \bY_{W,i}) \ge h(\bX_i) - h(\beta \circ \bX_i) ~\forall W,i.$
\end{enumerate}
Item (3) means: for all finite $V \subset \Gamma$ and $\epsilon>0$,
$$\lim_{(W,i)\to (G,\infty)} \ttheta_{W,i}\Big(\big\{(\sigma,\xi)\in \sym(m_i)^\Gamma\times A^{[m_i]}:~d_V((\sigma,\xi), \phi)<\epsilon \big\}\Big) =1$$
where $\bY_{W,i}$ is constructed from measures $\ttheta_{W,i}$ on $\sym(m_i)^\Gamma \times A^{[m_i]}$.

\end{lem}

\begin{proof}
Let $\teta_i$ be the probability measure on $\sym(m_i)^G \times A^{[m_i]}$ such that $\bX_i$ is constructed from $\teta_i$. Let $\ttheta_{W,i}$ be obtained by pushing $\teta_i$ forward under the map $J=J_{W,i}$ defined by
$$J(\sigma,\xi) =(S_{W,\sigma, \beta \circ \xi}, \xi).$$
Let $\bY_{W,i}$ be the approximate process over $\Gamma$ constructed from $\ttheta_{W,i}$. 

Since $\beta \circ \bX_i = \bX_{\beta,i}$, it follows that $\beta \circ \bY_{W,i}  $ is equivalent to $ \bY_{\beta,W,i}$.  The proofs of statements (2) and (3) are similar to the proofs of (1) and (2) of Lemma \ref{lem:lemma1}. 

To prove the last statement, fix $i$ and $W$. Let $\tbeta:\sym(m_i)^G \times A^{[m_i]} \to\sym(m_i)^G \times B^{[m_i]}$ be the map $\tbeta(\sigma,\xi)=(\sigma,\beta \circ \xi)$. Let $\teta_{i,\tbeta(\sigma,\xi)}$ be the fiber measure of $\teta_i$ over $\tbeta(\sigma,\xi)$. Thus 
$$\teta_i = \int \teta_{i,\tbeta(\sigma,\xi)}~d\tbeta_*\teta_i(\tbeta(\sigma,\xi)),\quad h(\bX_{i}) - h(\beta \circ \bX_{i}) = \frac{1}{m_i}\int H(\teta_{i,\tbeta(\sigma,\xi)})~d\tbeta_*\teta_i(\tbeta(\sigma,\xi)).$$
By abuse of notation, we also let $\tbeta$ denote the same map with $G$ replaced by $\Gamma$. We define the fiber measure $\ttheta_{W,i,\tbeta J(\sigma,\xi)}$ of $\ttheta_{W,i}$ over $\tbeta(J(\sigma,\xi))$ similarly. Thus
$$h(\bY_{W,i}) - h(\beta \circ \bY_{W,i}) = \frac{1}{m_i} \int H(\ttheta_{W,i,\tbeta J(\sigma,\xi)})~d\ttheta_{W,i}(J(\sigma,\xi)) = \frac{1}{m_i}\int H(\ttheta_{W,i,\tbeta J(\sigma,\xi)})~d\teta_i(\sigma,\xi).$$
Let $\alpha:\sym(m_i)^G\times A^{[m_i]}\to A^{[m_i]}$ denote the projection map. By abuse of notation, we also let $\alpha$ denote the projection map from $\sym(m_i)^\Gamma \times A^{[m_i]}$ to $A^{[m_i]}$. Because $\alpha J=\alpha$ and $J_*(\teta_{i,\tbeta J(\sigma,\xi)}) =\ttheta_{W,i,\tbeta J (\sigma,\xi)}$,
$$\alpha_*(\teta_{i,\tbeta J(\sigma,\xi)}) = \alpha_* (\ttheta_{W,i,\tbeta J (\sigma,\xi)}), \quad \forall (\sigma,\xi) \in \sym(m_i)^G\times A^{[m_i]}.$$

Since $\alpha$ is injective on the support of $\ttheta_{W,i,\tbeta J (\sigma,\xi)}$, it follows that
\begin{eqnarray*}
H( \ttheta_{W,i,\tbeta J (\sigma,\xi)} ) &=& H( \alpha_* \ttheta_{W,i,\tbeta J (\sigma,\xi)} ) =  H( \alpha_*\teta_{i,\tbeta J(\sigma,\xi)} ).
\end{eqnarray*}
We claim that
$$\teta_{i,\tbeta J(\sigma,\xi)} = \int \teta_{i,\sigma',\beta \xi'} ~d\teta_{i,\tbeta J(\sigma,\xi)}(\sigma',\xi').$$
To see this, suppose $X,Y,Z$ are any three Borel spaces, $\mu$ is a probability measure on $X$ and $\pi_1:X\to Y$, $\pi_2:Y\to Z$ are Borel maps. Then for $(\pi_2\pi_1)_*\mu$-a.e. $z\in Z$, $\mu_z = \int \mu_{\pi_1(x)} ~d\mu_z(x)$. This follows from the fact that if $\Sigma_Y$, $\Sigma_Z$ are the sigma-algebras on $X$ obtained from pulling back the sigma-algebras on $Y$, $Z$ respectively and $f$ is a bounded function on $X$ then $\bE[f|\Sigma_Z]=\bE[ \bE[f | \Sigma_Y] | \Sigma_Z]$. To see how this applies to the equation above, set $X=\sym(m_i)^G \times A^{[m_i]}$, $Y=\sym(m_i)^G \times B^{[m_i]}$ and $Z = \sym(m_i)^\Gamma \times B^{[m_i]}$. Let $\pi_1 = \tbeta$ and choose $\pi_2$ so that $\pi_2\pi_1 = \tbeta J$. 

Therefore,
$$\alpha_*\teta_{i,\tbeta J(\sigma,\xi)} = \int \alpha_* \teta_{i,\sigma',\beta \xi'} ~d\teta_{i,\tbeta J(\sigma,\xi)}(\sigma',\xi').$$
By concavity of entropy,
$$H( \alpha_* \teta_{i,\tbeta J(\sigma,\xi)} )\ge \int H(\alpha_* \teta_{i,\sigma',\beta \xi'} )~d\teta_{i,\tbeta J(\sigma,\xi)}(\sigma',\xi').$$
Because $\alpha$ is injective on the support of  $\teta_{i,\sigma',\beta \xi'}$, $H(\alpha_* \teta_{i,\sigma',\beta \xi'} ) = H(\teta_{i,\sigma',\beta \xi'} )$. Therefore,
$$H( \ttheta_{W,i,\tbeta J (\sigma,\xi)} )  \ge  \int H(\teta_{i,\sigma',\beta \xi'} )~d\teta_{i,\tbeta J(\sigma,\xi)}(\sigma',\xi').$$
Thus,
\begin{eqnarray*}
h(\bY_{W,i}) - h(\beta \circ \bY_{W,i})  &\ge&\frac{1}{m_i} \iint H(\teta_{i,\sigma',\beta \xi'})~d\teta_{i,\tbeta J(\sigma,\xi)}(\sigma',\xi') ~d\teta_i(\sigma,\xi)\\
&=&  \frac{1}{m_i} \int H(\teta_{i,\sigma,\beta \xi})~d\teta_i(\sigma,\xi)\\
&=& h(\bX_{i}) - h(\beta \circ \bX_{i}).
\end{eqnarray*}

 \end{proof}



We can now prove Proposition \ref{prop:orbit} whose statement is:
\vspace{0.2in}

\noindent {\bf Proposition  \ref{prop:orbit}}. {\it Let $G$ be an amenable group, $\bX=(T,X,\mu,\phi)$ an essentially free $G$-process with finite range $A$ and $S:(X,\mu)\to (X,\mu)$ be an essentially free measure-preserving Borel automorphism with the same orbits as $T$ (i.e., for $\mu$-a.e. $x\in X$, $\{T_gx:~g\in G\} = \{S^nx:~n\in \Z\}$).

Let $\beta:A \to B$ be a map and suppose the orbit change from $T$ to $S$ is measurable with respect to both the $T$-invariant sub-sigma-algebra generated by $\psi:=\beta\circ\phi$ and the $S$-invariant sub-sigma-algebra generated by $\psi$.  Then for any random sofic approximation $\sK$ to $G$,
$$\bh(\sK,\bX|\beta \circ \bX)= h(\sK, \bX| \beta \circ \bX  )  = h(S,\phi | \beta\circ \phi).$$ }

\begin{proof}
Let $\Gamma=\Z$. Let $\{\bX_i\}_{i=1}^\infty$ be a sequence of approximate processes adapted to $\sK'$, a subsequence of $\sK$, so that
\begin{enumerate}
\item $\lim_{i\to\infty} \bX_i = \bX$;
\item $\lim_{i\to\infty} h(\bX_i) - h(\beta \circ \bX_i) = \bh(\sK,\bX|\beta \circ \bX)$.
\end{enumerate}
Let $\bX_{\beta,i} = \beta \circ \bX_i$. Let $\{\bY_{\beta,W,i}\}_{i\in \N, W\subset G}$ be constructed as in the paragraph before Lemma \ref{lem:lemma0.5}. Let $\{\bY_{W,i}\}_{i\in \N, W\subset G}$ be the collection of approximate processes given by Lemma \ref{lem:forward}. 

A diagonalization argument implies that there exist increasing sequences $\{W_i\}_{i=1}^\infty$ and $\{k_i\}_{i=1}^\infty$ so that if $\bY_i:=\bY_{W_i,k_i}$ then $\lim_{i\to\infty} \bY_i = \bY=(S,X,\mu,\phi)$.  Moreover we may assume (using Lemma \ref{lem:lemma0.5}) that if $\sK^\rho:=\{\kappa^\rho_{W_i,k_i}\}_{i=1}^\infty$ then $\sK^\rho$ is a random sofic approximation to $\Z$. So Lemma \ref{lem:forward} implies
\begin{eqnarray*}
\bh(\sK,\bX|\beta \circ \bX) &=& \lim_{i\to\infty} h(\bX_i) - h(\beta \circ \bX_i)\le \limsup_{i\to\infty} h(\bY_{i}) - h(\beta \circ \bY_{i})\\
&\le& \bh(\sK^\rho, \bY|\beta \circ \bY) =  h(\bY| \beta \circ \bY).
\end{eqnarray*}
The last equality follows from Proposition \ref{prop:Zrelative}. Since $\bh(\sK,\bX|\beta \circ \bX) \ge h(\sK,\bX|\beta \circ \bX)$ it now suffices to prove  $h(\sK,\bX|\beta \circ \bX) \ge h(\bY| \beta \circ \bY)$. 

Let integers $m_i$ be given so that $\sK'=\{\kappa'_i\}_{i=1}^\infty$ where each $\kappa'_i$ is a probability measure on $\sym(m_i)^G$. Let $\sL=\{\lambda_i\}_{i=1}^\infty$ be a random sofic approximation to $\Z$ where each $\lambda_i$ is a probability measure on $\sym(m_i)^\Z$. By Proposition \ref{prop:Zrelative}, there exists a sequence $\{\bY'_{i}\}_{i \in \N}$ of approximate processes over $\Z$ adapted to $\sL'$, a subsequence of $\sL$ such that 
\begin{enumerate}
\item $\lim_{i \to \infty} \bY'_{i} = \bY$ strongly;
\item $\lim_{i\to \infty} h(\bY'_{i}) - h(\beta \circ \bY'_{i}) = h(\sL,\bY|\beta \circ \bY)= h(\bY|\beta \circ \bY)$.
\end{enumerate}
Let $\{\bX_{\beta,V,i}\}_{i\in \N, V\subset \Z}$ be constructed as in the paragraph before Lemma \ref{lem:lemma0.5} with the roles of $X$ and $Y$ swapped. This construction is possible because the orbit change is, by hypothesis, measurable with respect to the smallest sigma-algebra generated by $\psi$ and the $\Z$-action. Let $\{\bX_{V,i}\}_{i\in \N, V\subset \Z}$ be the collection of approximate processes given by Lemma \ref{lem:forward}. According to that lemma,
\begin{enumerate}
\item  $\lim_{(V,i) \to (\Z,\infty)} \bX'_{V,i}= \bX$ strongly.
\item $h(\bX'_{V,i}) - h(\beta \circ \bX'_{V,i}) \ge h(\bY'_{i}) - h(\beta \circ \bY'_{i}) ~\forall V,i.$
\end{enumerate}
So
\begin{eqnarray*}
h(\bY | \beta \circ \bY) &=& \lim_{i\to\infty} h(\bY'_{i}) - h(\beta \circ \bY'_{i})\le  \liminf_{(V,i) \to (\Z,\infty)} h(\bX'_{V,i}) - h(\beta \circ \bX'_{V,i}).
\end{eqnarray*}
A diagonalization argument and Lemma \ref{lem:lemma0.5} imply that there exist increasing sequences $\{V_i\}_{i=1}^\infty$ and $\{k_i\}_{i=1}^\infty$ so that if $\bX'_i:=\bX'_{V_i,k_i}$ then $\{\bX'_i\}_{i=1}^\infty$ is adapted to a random sofic approximation $\sK''$ of $G$. So the inequalities above imply
$$h(\bY | \beta \circ \bY) \le \limsup_{i\to\infty} h(\bX'_i) - h(\beta \circ \bX'_{i}) \le h(\sK'',\bX|\beta \circ \bX).$$
Observe that $\sK''=\{\kappa''_i\}_{i=1}^\infty$ where each $\kappa''_i$ is a probability measure on $\sym(n_i)^G$ and $\{n_i\}_{i=1}^\infty$ is a subsequence of $\{m_i\}_{i=1}^\infty$. It follows from Theorem \ref{thm:Zasymptotic} that $\sK''$ is asymptotic to a subsequence of $\sK'$. It follows from Theorem  \ref{thm:asymptotic2r} that
$$h(\sK'',\bX|\beta \circ \bX) \le h(\sK',\bX|\beta \circ \bX) \le h(\sK,\bX|\beta \circ \bX).$$
So $h(\sK,\bX|\beta \circ \bX) \ge h(\bY| \beta \circ \bY)$ as required.

\end{proof}






\end{document}